\documentclass[11pt,reqno]{amsart}

\usepackage[english]{babel}
\usepackage{amsmath}
\usepackage{amssymb}
\usepackage{amsthm}
\usepackage{latexsym}
\usepackage{mathtools}
\usepackage{thmtools}
\usepackage{amsfonts}

\usepackage[table]{xcolor}
\definecolor{red}{rgb}{1,0,0}

\usepackage{mathrsfs}
\usepackage{textcomp}
\usepackage[T1]{fontenc}
\usepackage{graphicx}
\usepackage{setspace}
\usepackage{nicefrac}
\usepackage{indentfirst}
\usepackage{enumerate}
\usepackage{wasysym}
\usepackage{upgreek}
\usepackage{booktabs}
\usepackage{etoolbox}
\usepackage{wrapfig}
\usepackage{mathdots}
\usepackage{floatflt}
\usepackage{parskip}
\usepackage{lscape}

\usepackage{ytableau}
\usepackage[enableskew]{youngtab}
\usepackage[]{hyphenat}
\usepackage{tikz}
\usetikzlibrary{tikzmark,decorations.pathreplacing}

\usepackage{tikz}
\usetikzlibrary{arrows,matrix}
\usetikzlibrary{positioning}
\usetikzlibrary{decorations}

\usepackage[all,cmtip]{xy}
\usepackage[labelfont=bf, font={small}]{caption}[2005/07/16]

\newcommand{\vvirg}{ , \dots , }

\newcommand{\calB}{\mathcal{B}}
\newcommand{\calC}{\mathcal{C}}

\newcommand{\calF}{\mathcal{F}}

\newcommand{\calH}{\mathcal{H}}
\newcommand{\calI}{\mathcal{I}}

\newcommand{\calP}{\mathcal{P}}

\newcommand{\calY}{\mathcal{Y}}
\newcommand{\calZ}{\mathcal{Z}}

\newcommand{\bbA}{\mathbb{A}}

\newcommand{\bbC}{\mathbb{C}}

\newcommand{\bbG}{\mathbb{G}}

\newcommand{\bbN}{\mathbb{N}}

\newcommand{\bbP}{\mathbb{P}}

\newcommand{\bbR}{\mathbb{R}}

\newcommand{\bbZ}{\mathbb{Z}}

\newcommand{\rmO}{\mathrm{O}}

\renewcommand{\phi}{\varphi}

\renewcommand{\theta}{\vartheta}

\newcommand{\dashto}{\dashrightarrow}

\renewcommand{\hat}[1]{\widehat{#1}}
\renewcommand{\bar}[1]{\overline{#1}}

\newcommand{\id}{\mathrm{id}}
\newcommand{\Id}{\mathrm{Id}}

\renewcommand{\Im}{\mathrm{Im} \;}  
\DeclareMathOperator{\codim}{codim}

\DeclareMathOperator{\End}{End}

\newcommand{\St}{\mathrm{St}}

\newcommand{\Mat}{\mathrm{Mat}}

\renewcommand{\HF}{\mathrm{HF}}
\newcommand{\HP}{\mathrm{HP}}

\newcommand{\fillwidthof}[3][c]
	{%
		\parbox
		{%
			\widthof{#2}%
		}%
		{%
			\ifx#1c%
				\centering#3%
			\else\ifx#1l%
				#3\hfill%
			\else\ifx#1r%
				\hfill#3%
			\fi\fi\fi%
		}%
	}%

\newcommand{\partinto}[2][]{
      \ifthenelse{\equal{#1}{}}{
      {{\scalebox{1.5}[.8]{$\; \vdash$} {#2} }}}{
      {{\underset{\hfill #1}{\scalebox{1.5}[.8]{$\; \vdash$}} {\; #2} }}
      }}
\renewcommand{\St}[2]{\mathrm{St}(#1,#2)}
\newcommand{\SO}[1]{\mathrm{SO}(#1)}
\newcommand{\rk}{\mathrm{rk}}
\newcommand{\vol}{\mathrm{vol}}

\usepackage[top=3cm,bottom=3cm,left=3cm,right=3cm]{geometry}

\newtheorem{theorem}{Theorem}[section]
 \newtheorem{corollary}{Corollary}[section]
 \newtheorem{lemma}{Lemma}[section]
 \newtheorem{proposition}{Proposition}[section]

  \theoremstyle{definition}
 \newtheorem{definition}{Definition}[section]
 \newtheorem{example}{Example}[section]

 \theoremstyle{remark}

 \renewcommand{\O}{\mathrm{O}}

\usepackage{diagbox}

\newcommand{\ybox}{\picture(1,1)
\put(0,0){\line(0,1){1}\line(1,0){1}}
\put(1,0){\line(0,1){1}}
\put(0,1){\line(1,0){1}}
\endpicture}

\newcommand{\ylabel}[1]{\makebox(1,1){#1}}
\newcommand{\ylabels}[1]{\makebox(1,1)[l]{#1}}
\newcommand{\pbinom}[2]{\begin{pmatrix} #1 \\ #2 \end{pmatrix}}


\author[T. Brysiewicz]{Taylor Brysiewicz}
\author[F. Gesmundo]{Fulvio Gesmundo}

\address[T. Brysiewicz]{Department of Mathematics, Texas A\&M University, College Station, TX 77843-3368, USA -- (current) Max Planck Institute for Mathematics in the Sciences, Inselstrasse 22, 04103 Leipzig, Germany}
\address[F. Gesmundo]{QMATH, Department of Mathematical Sciences, University of Copenhagen, Universitetsparken
5, 2100 Copenhagen O., Denmark --  (current) Max Planck Institute for Mathematics in the Sciences, Inselstrasse 22, 04103 Leipzig, Germany}

\email{taylor.brysiewicz@mis.mpg.de}
\email{fulvio.gesmundo@mis.mpg.de}
\newcommand{\rmSt}{\mathrm{St}}

\title{The Degree of Stiefel Manifolds}

\begin{document}

\begin{abstract}
 We compute the degree of Stiefel manifolds, that is, the variety of orthonormal frames in a finite dimensional vector space. Our approach employs techniques from classical algebraic geometry, algebraic combinatorics, and classical invariant theory.
\end{abstract}


\maketitle

\section{Introduction}

Frames are a generalization of bases of (real or complex) vector spaces, where one considers spanning sets that satisfy certain conditions. Formally, a collection of vectors $\{v_i\}_{i \in I}$ in a Hilbert space $\calH$ with inner product $\langle -,- \rangle$ is a \emph{frame} if there exist \emph{frame constants} $A,B \in \mathbb{R}_{>0}$ such that 
\[
A \Vert v \Vert ^2 \leq \sum_{i \in I} |\langle v, v_i \rangle|^2 \leq B \Vert  v \Vert ^2 
\]
where $ \Vert \cdot \Vert$ is the norm induced by the inner product. This set of inequalities is called the \emph{frame condition} and guarantees that $\{ v_i \}_{i \in I}$ spans $\calH$. If the set $I=\{1 \vvirg n\}$ is finite, then $\calH$ is finite dimensional and the frame $\{v_i\}_{i = 1 \vvirg n}$ is called a finite frame.

 A frame is called \emph{tight} if $A=B$ and \emph{Parseval} if $A=B=1$. Frames are extensively studied in linear algebra, functional analysis and operator theory. They find numerous applications in signal processing where they are used to represent signals in compact form while guaranteeing certain desired robustness properties \cite{Mal:WaveletSignalProcessing}.

From a computational point of view, a finite frame in $\mathbb{R}^k$ is encoded by a $k \times n$ matrix $\Phi$ whose columns are the coordinates of the frame vectors $\{v_i\}_{i= 1 \vvirg n}$. The corresponding frame is tight with frame constant $A$ if $\Phi\Phi^T = A\cdot \id_{k}$ and Parseval if $\Phi\Phi^T = \id_{k}$, where $\id_k$ denotes the $k \times k$ identity matrix. This characterizes all finite Parseval frames as the solutions of $\binom{k+1}{2}$ quadratic equations in the entries of a $k\times n$ matrix. In particular, it realizes the set of Parseval frames as an algebraic subvariety of the space of $k \times n$ matrices known as the \emph{Stiefel manifold}. We consider its Zariski closure $\St{k}{n}$ in the space of complex $k \times n$ matrices, $\Mat_{k \times n}(\mathbb{C})$. 

Equivalently, Stiefel manifolds can be realized as collections of $k$ orthonormal vectors in an $n$-dimensional vector space, recorded by the rows of the matrix $\Phi$. In this setting, if $k < n$, the variety $\St{k}{n}$ is naturally identified with the homogeneous space $\SO{n}/\SO{n-k}$ where $\SO{n-k}$ is regarded as the stabilizer of $k$ fixed (complex) orthonormal vectors (see Section \ref{section: preliminaries}). This perspective allows for the use of powerful tools from representation theory and classical invariant theory in the study of Stiefel manifolds.
 
Long standing open problems in finite frame theory have been recently solved by understanding spaces of frames as embedded algebraic varieties  \cite{CahMixStr:ConnIrredFUNTF,NeeSho:SymplGeomConnFrames,Vinz:SmallFrameCertificateInjectivity}. Nonetheless, one of the fundamental invariants of an embedded variety, its \emph{degree}, remains unknown for almost all spaces of frames. When $n=k$, the Stiefel manifold $\St{k}{n}$ coincides with the orthogonal group $\O(n)$ and its degree as a subvariety of the space of $n \times n$ matrices was computed in \cite{BBBKR:DegreeSOn}. The main purpose of this paper is to compute the degree of Stiefel manifolds in general.

\begin{theorem}\label{thm: main theorem}
 Let $n \geq k$. 
\begin{itemize}
 \item Suppose $n \leq 2k-1$ and write $n = 2r$ or $n = 2r+1$ depending on the parity. Then
 \[
  \deg \St{k}{n} = 2^k   \cdot L_{k,n} 
  \]
  where $L_{k,n}$ denotes the number of collections of non-intersecting lattice paths from the points $$A = \{ (-a_i,0) : i=1 \vvirg r\}\quad \text{ to } \quad B = \{(0,b_j) : k = 1 \vvirg r\}$$ with
  \begin{align*}
(a_1 \vvirg a_r) &= ( \underbrace{ k -1, k -2, \vvirg k-(n-k)}_{n-k} , \underbrace{2k-n-2, 2k-n-4  \vvirg n-2r}_{r-(n-k)} ), \\
(b_1 \vvirg b_r) &= ( n -2, n -4, \vvirg n - 2r ).
  \end{align*}
 \item Suppose $n \geq 2k-1$. Then
 \[
  \deg \St{k}{n} = 2^{\binom{k+1}{2}}.
  \]
 \end{itemize}
\end{theorem}

While this theorem gives a combinatorial interpretation to the degrees of Stiefel manifolds, a bijective proof remains elusive. Such a proof amounts to establishing a bijection between intersection points of suitable linear sections of the Stiefel manifolds and collections of non-intersecting lattice paths. This may give a simpler and more direct proof of Theorem \ref{thm: main theorem}. Moreover, it would bolster the use of homotopy methods for studying Stiefel manifolds and their subvarieties. Indeed, an explicit bijection immediately gives a representation of any Stiefel manifold via a \emph{witness set}, the fundamental data type in numerical algebraic geometry \cite{BertiniBook}.

\section{Preliminaries}\label{section: preliminaries}

\subsection{Degree, Hilbert function and Hilbert polynomial}
We introduce some basic notions about the degree of algebraic varieties. The material of this section is classical and we refer to \cite[Lecture 18]{Harris:AlgGeo} and \cite[Section I.1.9]{Eis:CommutativeAlgebra} for formal definitions and an exposition of the theory. We include some basics here for the reader's convenience and to introduce some notation and convention.

We use homogeneous coordinates $x_0 \vvirg x_N$ on the projective space $\bbP^N = \bbP_\bbC^N$. The affine space $\bbA^N$ is identified with the affine chart $\{ x_0 \neq 0\}$ of $\bbP^N$ and its complement $H_\infty = \{ x_0 = 0\}$ is called the \emph{hyperplane at infinity}.

A \emph{variety} is an affine or projective algebraic variety, reduced and possibly reducible. If $X \subseteq \bbA^N$ is affine, write $\bar{X}$ for its closure in $\bbP^N$. We denote by $I_X$ the defining ideal of $X$, which is an ideal in the polynomial ring $\bbC[x_1 \vvirg x_N]$ or $\bbC[x_0\vvirg x_N]$ depending on whether $X$ is affine or projective. Write $\bbC[X]$ for the coordinate ring of $X$, that is, the quotient of the polynomial ring over $I_X$. When $X$ is projective (resp. affine), the natural grading of the polynomial ring induces a grading (resp. filtration) on $\bbC[X]$.

If $X \subseteq \bbA^N$ (resp. $X \subseteq \bbP^N$) is an irreducible variety of dimension $n$, the degree of $X$, denoted $\deg(X)$, is the number of points of intersection of $X \cap L$ where $L$ is a generic linear space of codimension $n$. If $X$ is possibly reducible but equidimensional, then the degree of $X$ is the sum of the degrees of its irreducible components. If $X$ is possibly reducible and possibly not equidimensional, then the degree of $X$ is the degree of the union of the components of largest dimension. It is immediate that $\deg(X) = \deg(\bar{X})$.

Fix a projective variety $X$ of codimension $c$ and suppose $I_X$ is generated by $c$ homogeneous polynomials $f_1 \vvirg f_c$ of degree $d_1 \vvirg d_c$ respectively. Then $\deg(X) = d_1 \cdots d_c$ and $X$ is called a \emph{complete intersection}. More generally, for any variety $X$ of codimension $c$, the ideal $I_X$ is generated by at least $c$ homogeneous polynomials: the product of their degrees is called the \emph{B\'ezout bound} and always serves as an upper bound for $\deg(X)$.

The \emph{Hilbert function} of $X$ is the function $\HF_X : \bbN \to \bbN$, defined by $\HF_X( t) = \dim (\bbC[X])_{\leq t}$ or $\HF_X (t) = \dim \bbC[X]_{t}$ depending on whether $X$ is affine or projective. The Hilbert function is eventually a polynomial: there exists a univariate polynomial $\HP_X(t)$, called the \emph{Hilbert polynomial} of $X$, with the property that $\HF_X(t) = \HP_X(t)$ for $t \gg 0$. Moreover, the degree of $\HP_X$ is $\dim X$ and its leading coefficient is $\dfrac{\deg(X)}{ \dim (X) !}$.

Given a polynomial $f \in \bbC[x_1 \vvirg x_N]$, write $\hat{f}$ for its homogenization via $x_0$, i.e. the unique homogeneous polynomial in $\bbC[x_0 \vvirg x_N]$ with $\deg(f) = \deg(\hat{f})$ such that $ {\hat{f}}|_{x_0 = 1} = f$. If $X$ is an affine variety and $f_1 \vvirg f_\ell$ are generators of its ideal $I_X$ then ${\hat{f}_1} \vvirg {\hat{f}_\ell}$ cut out a scheme in $\bbP^N$ which is possibly not reduced; we call this scheme the \emph{naive homogenization} of $X$ (with respect to the chosen generators). We have the following elementary fact.
\begin{lemma}\label{lemma: naive homogenization}
 Let $X \subseteq \bbA^N$ be an affine variety and let $f_1 \vvirg f_\ell$ be generators of $I_X$. Let $Y \subseteq \bbP^N$ be the naive homogenization of $X$. The irreducible components of $\bar{X}$ are irreducible components of $Y$ and every other irreducible component of $Y$ is supported on $H_\infty$.
\end{lemma}
\begin{proof}
Clearly $\bar{X} \subseteq Y$. It suffices to show that $Y \cap \{ x_0 \neq 0\} \subseteq X$. But localizing the equations of $Y$ at $x_0 \neq 0$, one obtains exactly  $f_1 \vvirg f_\ell$, which are defining equations for $X$.
\end{proof}
\begin{corollary}\label{corol: small compts at infinity implies same degree}
 Let $X \subseteq \bbA^N$ be an affine variety and let $f_1 \vvirg f_\ell$ be generators of $I_X$. Let $Y \subseteq \bbP^N$ be the naive homogenization of $X$. If all irreducible components of $Y \cap H_\infty$ have dimension strictly smaller than $\dim X$, then $\deg(X) = \deg(Y)$. 
\end{corollary}

\subsection{Orbits, algebraic groups and semistable points}

We state the Algebraic Peter--Weyl Theorem \cite[Thm. 4.2.7]{GooWal:Symmetry_reps_invs} in full generality for a complex semisimple algebraic group and we describe the application to the special orthogonal group that will be needed in Section \ref{section: repdegree}. Our references for this material are \cite{GooWal:Symmetry_reps_invs,FulHar:RepTh}.

Let $G$ be a complex semisimple algebraic group. Fix a maximal torus $T \subseteq G$ and a Borel subgroup $B$. Denote by $\Lambda$ the weight lattice of $G$ with respect to $T$ and by $\Lambda_+$ the cone of dominant weights with respect to $B$. In other words, $\Lambda_+ = \Lambda \cap W$ where $W$ is the principal Weyl chamber. For a dominant weight $\lambda$, denote by $V_\lambda$ the irreducible representation with highest weight $\lambda$. We point out that if $G$ is not simply connected, then there are dominant weights not corresponding to an irreducible representation of $G$. Denote by $\Lambda_+^G$ the set of integral dominant weights corresponding to representation of $G$.

Fix a $G$-representation $V$ (not necessarily irreducible). Given $w \in V$, let $G_w = \{ g \in G : g\cdot w = w\}$ be the stabilizer of $w$ in $G$, which is a closed subgroup of $G$. An element $w \in V$ is called \emph{semistable} (for the action of $G$) if the orbit $G \cdot w \subseteq V$ is Zariski closed (equivalently Euclidean closed). The set $G \cdot w$ is naturally an abstract algebraic variety $G \cdot w \simeq G/ G_w$, where $G/G_w$ denotes the set of left cosets of $G_w$ in $G$.

In this case, the affine coordinate ring of $G \cdot w$ can be written intrinsically in terms of the representation theory of $G$ and $G_w$, via the Algebraic Peter--Weyl Theorem:
\begin{equation}\label{eq: peter weyl general}
 \bbC[G \cdot w] = \bigoplus _{\lambda \in \Lambda_+^G} V_\lambda \otimes [V_\lambda^*]^{G_w} 
\end{equation}
where $[V_\lambda^*]^{G_w}$ denotes the subspace of $G_w$-invariants in $V_\lambda^*$.

Our goal is to apply the Algebraic Peter--Weyl Theorem to compute the leading coefficient of the Hilbert polynomial of Stiefel manifolds. In general, it is not immediate how the grading of the polynomial ring $\bbC[V]$ descends to a filtration of $\bbC[G \cdot w]$. However, if $G$ can be realized as a closed subgroup of the endomorphism space of $V$, we have the following result.
\begin{lemma}\label{lemma: appearing weights}
 Let $G$ be a semisimple algebraic group, let $V$ be a faithful $G$-representation such that the image of $G$ in $\End(V)$ is closed. Let $w \in V$ be a semistable point. For every dominant weight $\lambda$ of $G$, the summand $V_\lambda \otimes [V_\lambda^*]^{G_w}$ appears in $\bbC[G \cdot w]_{\leq j}$ if and only if $\lambda \in j\calC_V$ where $\calC_V$ is the convex hull of the integral weights occurring in $V$.
\end{lemma}
\begin{proof}
Since $V$ is faithful, we may regard $G$ as a closed subvariety of $\End(V)$. Regard $\End(V) \simeq V^{\oplus \dim V}$ as a $G$-representation with respect to the left-composition by elements of $G$: the integral weights occurring in $V$ are the same as the integral weights occurring in $\End(V)$. The statement holds for $\bbC[G]$ regarded as a quotient of $\bbC[\End(V)]$ from the \emph{Claim} in the proof of \cite[Theorem 9.1]{DerkKraft:ConstructiveInvTheory}. 

Now, consider the linear map 
\begin{align*}
 \End(V) &\to V \\
 L &\mapsto Lw.
\end{align*}
By linearity, the pullback map on coordinate rings $\bbC[V] \to \bbC[\End(V)]$ preserves the grading. A consequence is that the restricted map $G \to G \cdot w$ defined by $g \mapsto g\cdot w$ induces a pullback map on coordinate rings $\bbC[G\cdot w] \to \bbC[G]$ which preserves the filtration given by the grading of the polynomial ring. In particular $\bbC[G \cdot w] _{\leq j }$ is mapped to $\bbC[G]_{\leq j}$.  This concludes the proof.
\end{proof}

\subsection{Representation theory of $\SO{n}$ and branching rules}
We briefly review some basics of the representation theory of $\SO{n}$. We refer to \cite{GooWal:Symmetry_reps_invs,FulHar:RepTh} for an exposition of the theory and to \cite[LaPlanche II, IV]{Bourb:LieGpsLieAlg-46} for the explicit numerical data.

When $n = 2r+1$ is odd, then $\SO{n}$ has dimension $\binom{n}{2}$ and rank $r$. Let $e_1 \vvirg e_r$ be the simple weights. The fundamental weights are $\omega_i = e_1 + \cdots + e_i$ for $ i = 1 \vvirg r-1$ and $\omega_r = \frac{1}{2} (e_1 + \cdots + e_r)$; in particular $\omega_r$ does not provide a representation for $\SO{n}$. The integral cone $\Lambda^{\SO{n}}_+$ is given by the $\bbZ_+$-linear combinations of $\omega_1 \vvirg \omega_{r-1}$ and $2\omega_r$. Equivalently integral linear combinations of the fundamental weights are recorded as partitions $\lambda = (\lambda_1 \vvirg \lambda_r)$ where $\lambda_j$ is the coefficient of $e_j$ in the linear combination. In summary, the irreducible representations of $\SO{2r+1}$ are uniquely determined by a partition of length $r$, that is, an integer sequence $\lambda_1 \geq \cdots \geq \lambda_r \geq 0$.

When $n = 2r$ is even, then $\SO{n}$ has dimension $\binom{n}{2}$ and rank $r$. Let $e_1 \vvirg e_r$ be the simple weights. The fundamental weights are $\omega_i = e_1 + \cdots + e_i$ for $ i = 1 \vvirg r-2$, $\omega_{r-1} = \frac{1}{2}(e_1 + \cdots + e_{r-1}+e_r)$ and $\omega_r = \frac{1}{2}(e_1 + \cdots + e_{r-1}-e_r)$; in particular $\omega_{r-1}$ and $\omega_r$ do not provide representations for $\SO{n}$. The integral cone $\Lambda^{\SO{n}}_+$ is given by the $\bbZ_+$-linear combinations of $\omega_1 \vvirg \omega_{r-2}$, $\omega_{r-1} + \omega_r$ and $\omega_{r-1}-\omega_r$. Equivalently, integral linear combinations of the fundamental weights are recorded as non-increasing sequences $\lambda = (\lambda_1 \vvirg \lambda_r)$ with $\lambda_r$ possibly negative. In summary, the irreducible representations of $\SO{2r}$ are uniquely determined by non-increasing integral sequences $\lambda_1 \geq \vvirg \geq \lambda_{r-1} \geq | \lambda_r|$.

Moreover, it is immediate that for every dominant weight $\lambda$ for $\SO{n}$ we have $V_\lambda \simeq V_\lambda^*$ as $\SO{n}$-representations, and the identification is simply given via contraction with the quadratic form. 

We describe the branching rules for the restriction of representations from $\SO{n}$ to $\SO{n-1}$, realized as the subgroup stabilizing a fixed hyperplane. See \cite[Section 8.3]{GooWal:Symmetry_reps_invs}.
\begin{lemma}[Branching Rules]\label{lemma:branchingRules}
Let $ n = 2r$ be even. Let $\lambda=(\lambda_1 \vvirg \lambda_r)$ be a dominant integral weight for $\SO{2r}$. Then, as an $\SO{n-1}$-representation, $V_\lambda$ reduces to $V_\lambda = \textstyle \bigoplus_{\mu \in \calI} W_\mu $ where $\calI$ ranges over all dominant integral weights $\mu = (\mu_1 \vvirg \mu_{r-1})$ of $\SO{2r-1}$ such that
\[
\lambda_1 \geq \mu_1 \geq \lambda_2 \geq \mu_2 \geq \cdots \geq \lambda_{r-1} \geq \mu_{r-1} \geq |\lambda_r|.
\]

Let $n = 2r+1$ be odd. Let $\lambda=(\lambda_1 \vvirg \lambda_r)$ be a dominant integral weight for $\SO{2r+1}$. Then as an $\SO{n-1}$-representation, $V_\lambda$ reduces to $V_\lambda = \textstyle \bigoplus_{\mu \in \calI} W_\mu $ where $\calI$ ranges over all dominant integral weights $\mu = (\mu_1 \vvirg \mu_{r})$ of $\SO{2r}$ such that
\[
\lambda_1 \geq \mu_1 \geq \lambda_2 \geq \mu_2 \geq \cdots \geq \lambda_{r-1} \geq \mu_{r-1} \geq \lambda_r \geq |\mu_r|.
\]
\end{lemma}

\subsection{Stiefel manifolds}\label{section: prelim stiefel}

This section is devoted to classical results about Stiefel manifolds. For $ k \leq n$, define
\[
\St{k}{n} = \{ A \in \Mat_{k \times n} (\bbC) : AA^T = \id_k\}. 
\]
This is an affine variety whose defining equations are the $\binom{k+1}{2}$ quadrics given by the entries of the symmetric $k \times k$ matrix $AA^T - \id_k$.
The special orthogonal group $\SO{n}$ acts on $\St{k}{n}$ by right multiplication: indeed if $A \in \St{k}{n}$ and $g \in \SO{n}$, we have $(Ag)(Ag)^T = Agg^T A^T = A \id_n A^T = \id_k$. Note that if $n = k$, then $\St{k}{n}$ coincides with the orthogonal group $\rmO(n)$: in particular $\St{n}{n}$ is reducible. 

If $k < n$, then the action of $\SO{n}$ on $\St{k}{n}$ is transitive making $\St{k}{n}$ the orbit of $\SO{n}$ under this action: if $A,B \in \St{k}{n}$ then the rows of $A$ are orthonormal as well as the rows of $B$; since $k< n$, there exists an element $g \in \SO{n}$ sending the rows of $A$ to the rows of $B$. Observe that the stabilizer of $A \in \St{k}{n}$ under this action is the subgroup acting as the identity on the space spanned by the rows of $A$: this is a conjugate of the subgroup $\SO{n-k}$.

We deduce the following classical fact.
\begin{lemma}
If $k < n$, then $\St{k}{n}$ is irreducible and isomorphic to the homogeneous space $\SO{n} / \SO{n-k}$. In particular, $\St{k}{n}$ is smooth, irreducible, reduced and $$\dim \St{k}{n} = \binom{n}{2} - \binom{n-k}{2}.$$
\end{lemma} 

Thus, the codimension of $\St{k}{n}$ in $\Mat_{k\times n}$ is $nk - (\binom{n}{2} - \binom{n-k}{2}) = \binom{k+1}{2}$, the same as the number of quadrics defining it. As a consequence, we obtain that $\St{k}{n}$ is affinely cut out by these $\binom{k+1}{2}$ quadrics.

\subsection{Outline of the Proof of the Main Theorem}
The proof of Theorem \ref{thm: main theorem} is essentially divided in two parts. 

The first part is purely geometric and pertains to the green entries in Table \ref{table: degree table}. We compute the degree of $\St{k}{n}$ when $n \geq 2k-1$. In this case, the naive homogenization of $\St{k}{n}$ coincides with its closure in projective space $\bar{\St{k}{n}}$, so that $\bar{\St{k}{n}}$ is a complete intersection and its degree equals the B\'ezout bound. The proof relies on a dimension argument, showing that the naive homogenization of $\St{k}{n}$ does not have additional components at infinity in this range. This is the result of Theorem \ref{thm: dim Zinfty} and Theorem \ref{thm: stiefel in bezout bound}.

As noticed in Section \ref{section: prelim stiefel}, when $n=k$, we have $\rmSt(k,n) = \O(n)$. The degrees of the orthogonal groups were determined in \cite{BBBKR:DegreeSOn} and appear in Table \ref{table: degree table} in dark blue.

The rest of the proof is aimed at determining the degrees of $\rmSt(k,n)$ for $k+1 \leq n \leq 2k-2$ which appear in Table \ref{table: degree table} as light blue. In this case the degree of $\St{k}{n}$ is determined by computing the leading coefficient of its Hilbert polynomial. We apply a representation theoretic argument, built on the Algebraic Peter--Weyl Theorem, to the homogeneous space $\SO{n} / \SO{n-k}$. Determining the dimensions of the summands of \eqref{eq: peter weyl general} is difficult. Following the work of \cite{Kaz:NewtonPoly,Brion:ImageApplMoment,Brion:IntroActionsAlgGps,DerkKraft:ConstructiveInvTheory} in the setting of spherical varieties and generic orbits, we reduce the calculation of $\deg \St{k}{n}$ to an integral of certain alternating functions, arising from volumes of \emph{Gelfand--Tsetlin polytopes} associated to the representations of the orthogonal group and its invariant spaces. The proof is performed by an inductive argument which allows us to compute volumes of Gelfand--Tsetlin polytopes as alternating polynomials in the entries of their top row, see Theorem \ref{thm: volume GT general}. The base cases for induction are given by the entries of Table \ref{table: degree table} in dark green and the induction step moves south-east in the table. The degree formula for the degree of $\St{k,n}$ in this range is given in Theorem \ref{thm: degformula}, and its expression in terms of the combinatorics of non-intersecting lattice paths is obtained in Corollary \ref{corol: degreebylatticepaths}.
\begin{table}
\begin{center}
\begin{tabular}{|l|c|c|c|c|c|c|c|c|c|c|}\hline
$k \backslash n$ & 1 & 2 & 3 & 4 & 5 & 6 &7 &8&9&10\\
\hline
1&\cellcolor{green}2 &\cellcolor{green!21}2&\cellcolor{green!21}2&\cellcolor{green!21}2&\cellcolor{green!21}2& \cellcolor{green!21}2&\cellcolor{green!21}2&\cellcolor{green!21}2&\cellcolor{green!21}2&\cellcolor{green!21}2
\\
\hline
2&*&4\cellcolor{blue!50}&\cellcolor{green}8&\cellcolor{green!21}8&\cellcolor{green!21}8&\cellcolor{green!21}8&\cellcolor{green!21}8&\cellcolor{green!21}8&\cellcolor{green!21}8&\cellcolor{green!21}8
\\
\hline
3&*&*&16\cellcolor{blue!50}&40\cellcolor{blue!25}&\cellcolor{green}64&\cellcolor{green!21}64&\cellcolor{green!21}64&\cellcolor{green!21}64&\cellcolor{green!21}64&\cellcolor{green!21}64
\\
\hline
4&*&*&*&80\cellcolor{blue!50}&384\cellcolor{blue!25}&704\cellcolor{blue!25}&\cellcolor{green}1024&\cellcolor{green!21}1024&\cellcolor{green!21}1024&\cellcolor{green!21}1024
\\
\hline
5&*&*&*&*&768\cellcolor{blue!50}&4768\cellcolor{blue!25}&14848\cellcolor{blue!25}&23808\cellcolor{blue!25}&\cellcolor{green}32768&\cellcolor{green!21}32768
\\
\hline
6&*&*&*&*&*&9536\cellcolor{blue!50}&111616\cellcolor{blue!25}&420736\cellcolor{blue!25}&1064960\cellcolor{blue!25}&1581056\cellcolor{blue!25}
\\
\hline
7&*&*&*&*&*&*&\cellcolor{blue!50}223232&\cellcolor{blue!25}3433600 &22429696\cellcolor{blue!25}&66082816\cellcolor{blue!25}
\\
\hline
8&*&*&*&*&*&*&*&\cellcolor{blue!50} 6867200&\cellcolor{blue!25}196968448&1604859904\cellcolor{blue!25}
\\
\hline
9&*&*&*&*&*&*&*&*&\cellcolor{blue!50}393936896&\cellcolor{blue!25}14994641408 
\\
\hline
10&*&*&*&*&*&*&*&*&*&\cellcolor{blue!50}29989282816
\\
\hline
\end{tabular}
\caption{Degrees of Stiefel manifolds: (light green) Theorem \ref{thm: stiefel in bezout bound}, (dark blue)  $\deg(\O(n))$ computed in \cite{BBBKR:DegreeSOn}, (light blue) Theorem \ref{thm: degformula}, (dark green) base of induction for proof of Theorem \ref{thm: degformula}}\label{table: degree table}
\end{center}
\end{table}

\section{Degree of $\rmSt(k,n)$ for $n \geq 2k-1$}\label{section: bezout range}

In this section we prove the first part of Theorem \ref{thm: main theorem} when $n \geq 2k-1$. Regard the space $\Mat_{k \times n}$ as the open subset of $\bbP ( \Mat_{k \times n} \oplus \bbC)$ and let $z_0$ be a coordinate on the direct summand $\bbC$, so that $\Mat_{k\times n}$ is regarded as the principal open set $\{ z_0 \neq 0 \}$ and $H_\infty = \{ z_0 = 0\}$ is the hyperplane at infinity.

Let $\bar{\St{k}{n}}$ be the closure of $\St{k}{n}$ in $\bbP (\Mat_{k \times n} \oplus \bbC)$ and let 
\[
\calZ(k,n) = \{ (A,z_0) \in \bbP (\Mat_{k \times n} \oplus \bbC) :  AA^T - z_0^2 \Id_k = 0\}
\]
be the naive homogenization of $\St{k}{n}$. Let 
\[
\calZ_\infty (k,n) = \calZ(k,n) \cap H_\infty = \{ A \in \bbP \Mat_{k \times n} : AA^T = 0\}. 
\]
First, we compute $\dim  \calZ_{\infty}(k,n)$ following a standard argument via an incidence correspondence over the Fano scheme of the quadric hypersurface. This is similar to the classical argument for determinantal varieties as in \cite[II.2]{ArCoGrHa:Vol1}.

 Given a variety $X \subseteq \bbP V$, denote the \emph{Fano scheme of $s$-planes in $X$} is $$\calF_{s}(X) = \{ E \in \bbG(s,V) : \bbP E \subseteq X\},$$ where $\bbG(s,V)$ denotes the Grassmannian of $s$-planes in $V$.
Let $q_n = x_1 ^2 + \cdots + x_n^2$ and let $Q_n = \{ q_n = 0\} \subseteq \bbP^{n-1}$ be the corresponding quadric hypersurface.

\begin{lemma}\label{lemma: zero iff image in quadric}
Let $A \in \Mat_{k\times n}$, then 
\[
  AA^T = 0 \quad \text{ if and only if } \quad \Im A^T \subseteq Q_n.
 \]
 In particular, $\calZ_{\infty}(k,n) = \{ A \in \Mat_{k \times n} : \Im A^T \in \calF_{\rk(A)}(Q_n)\}$.
\end{lemma}
\begin{proof}
Suppose $A A^T = 0$ and let $v \in \Im A^T$, with $v = A^T c$ for some $c \in \bbC^k$. Then $q_n(v) =  v^T v = c^T A A^Tc = 0$. Conversely, suppose $q_n(v) = 0$ for every $v \in \Im A^T$, so that  $0 = q_n(A^T c) = c^TAA^Tc$ for every $c \in \bbC^k$. This implies that the quadratic form associated to $AA^T$ is identically $0$ or equivalently $AA^T = 0$. 
\end{proof}

If $s \leq n/2$, then the dimensions of the Fano schemes associated to the quadric are given by
\[
 \dim \calF_{s}(Q_n) = ns - \frac{1}{2} (3 s^2 +s).
\]
If $s > n/2$ then $\calF_{s}(Q_n) = \emptyset$. We refer to \cite[\S6.1]{GrifHar:PrinciplesAlgebraicGeometry} for the proof and additional information on the geometry of the Fano scheme. 

\begin{theorem}\label{thm: dim Zinfty}
For every $k, n$, we have 

\[
 \dim \calZ_{\infty}(k,n) = \left\{ \begin{array}{ll}
                               \frac 1 8 (n^2+4kn-2n-8) & \text{if $n < 2k-1$ and $n$ is odd}\\
                               \frac{1}{8}(n^2+4kn-4k-9) & \text{if $n < 2k-1$ and $n$ is even}\\
                               \binom{n}{2} - \binom{n-k}{2} -1 & \text{if $n \geq 2k-1$ and $n$ is even}
                              \end{array}\right.
\]

In particular, if $n \geq 2k-1$, then $\dim \calZ_{\infty}(k,n) = \dim \rmSt(k,n) -1$.
\end{theorem}
\begin{proof}
Let $s_{max}  = \min \{ k, \lfloor n/2 \rfloor \}$. For every $s \leq s_{max}$ define 
\[
 \xymatrix{ &\calY_s = \{ (A, E) \in \bbP \Mat_{k \times n} \times \calF_s(Q_n) : \Im A^T \subseteq E \} \ar[ld]^{\pi_1} \ar[rd]_{\pi_2} \\
 \bbP \Mat_{k \times n} & & \calF_s(Q_n)
 }
\]
where $\pi_1,\pi_2$ are the natural projections on the first and second factor. The generic fiber of $\pi_2$ over $E$  is
\[
{ \calY_s}_E := \{ A \in \bbP \Mat_{k\times n} : \Im A^T \subseteq E \} \subseteq \bbP \Mat_{k\times n}
\]
which is a (projective) linear space of dimension $ks -1$. The Theorem of the Dimension of the Fibers \cite[Section I.6.3]{Shaf:BasicAlgGeom1} provides
\[
\dim \calY_s = \dim \calF_s(Q_n) +  {\dim \calY_s}_E = ns - \frac{1}{2} (3 s^2 +s) + ks - 1.
\]

By Lemma \ref{lemma: zero iff image in quadric},  $\calZ_\infty(k,n) = \bigcup _{s = 1} ^{s_{max}} \pi_1(\calY_s)$ and the projection $\pi_1$ is generically one-to-one. This shows that 
\[
 \dim \calZ_\infty(k,n) = \max \left\{ ns - \frac{1}{2} (3 s^2 +s) + ks - 1 : s = 1 \vvirg s_{max} \right\}.
\]
Rewrite $\dim \calY_s = s ( n+k - \frac{1}{2} - \frac{3}{2}s) -1$. As a function of $s$, $\dim \calY_s$ is increasing between $0$ and $\frac{n+k-1/2}{3}$. In particular, $\dim \calY_s$ is increasing on $0 \leq s \leq s_{max}$ whenever $n>2k$ or $n<2k-1$, therefore the maximum value (on an integer) of $\dim \calY_s$ in this range is attained at $s_{max}$. For the remaining two cases of $(k,2k-1)$ and $(k,2k)$, one can check that the same conclusion holds. We obtain 
\[
 \dim \calZ_{\infty}(k,n) = \dim \calY_{s_{max}} = \left\{ \begin{array}{ll}
                               \frac 1 8 (n^2+4kn-2n-8) & \text{if $n < 2k-1$ and $n$ is odd}\\
                               \frac{1}{8}(n^2+4kn-4k-9) & \text{if $n < 2k-1$ and $n$ is even}\\
                               \binom{n}{2} - \binom{n-k}{2} -1 & \text{if $n \geq 2k-1$ and $n$ is even}
                              \end{array}\right.
\]
which concludes the proof.
\end{proof}

A consequence of Theorem \ref{thm: dim Zinfty} is that, when $n \geq 2k-1$, $\calZ_\infty(k, n)$ does not contain irreducible components of $\calZ(k,n)$ of dimension as large as $\dim \rmSt(k,n)$. In fact, $\calZ_\infty(k,n)$ does not contain irreducible components of $\calZ(k,n)$ at all. As a consequence, we obtain,
\begin{theorem}\label{thm: stiefel in bezout bound}
 If $n \geq 2k-1$, then $\bar{\rmSt (k,n)} = \calZ(k,n)$ is a complete intersection of $\binom{k+1}{2}$ quadrics. In particular, $$\deg \rmSt(k,n) = 2^{\binom{k+1}{2}}.$$
\end{theorem}
\begin{proof}
The equations defining $\calZ(k,n)$ are the entries of $AA^T - z_0^2 \Id_k = 0$. Since $AA^T - z_0\Id_k$ is symmetric, there are at most $\binom{k+1}{2}$ linearly independent equations. Therefore, every irreducible component of $\calZ(k,n)$ has codimension at most $\binom{k+1}{2}$. 

Since $\dim \rmSt(k,n) = \dim \SO{n} - \dim \SO{n-k} = \binom{n}{2} - \binom{n-k}{2}$, by Theorem \ref{thm: dim Zinfty}, we have $\dim \calZ(k,n) = \binom{n}{2} - \binom{n-k}{2}$ as well, so that $\codim \calZ(k,n) = nk - \left[ \binom{n}{2} - \binom{n-k}{2} \right] = \binom{k+1}{2}$. 

This shows that $\calZ(k,n) = \bar{\rmSt(k,n)}$ and in particular it is a complete intersection of the quadrics defined by $AA^T - z_0\Id_k$. By B\'ezout's theorem, we conclude $\deg \calZ(k,n) =  \deg \rmSt(k,n) = 2^{\binom{k+1}{2}}$.
\end{proof}

\section{Degree of $\St{k}{n}$ when $n \leq 2k-1$}\label{section: repdegree}

Theorem \ref{thm: dim Zinfty} shows that when $k \leq  n < 2k-1$, the variety $\calZ(k,n)$ has components at infinity of dimension at least as large as $\dim \bar{\St{k}{n}}$. Therefore, $\deg \St{k}{n}$ is not equal to the B\'ezout bound in these cases.

In this range, we compute the degree by computing the leading coefficient of the Hilbert polynomial of $\St{k}{n}$ via the Algebraic Peter--Weyl Theorem. More precisely, we use
\begin{equation}\label{eq: deg as limit}
\deg \St{k}{n}= N!\lim_{j \to \infty} \frac{\dim \bbC[\St{k}{n}]_{\leq j}}{j^N}
\end{equation}
where $N = \dim \St{k}{n} = \binom{n}{2} - \binom{n-k}{2}$.

The values of $\dim \bbC[\St{k}{n}]_{\leq j}$ will be computed via Lemma \ref{lemma: appearing weights}. Indeed, \eqref{eq: peter weyl general} provides
\[
 \bbC[\St{k}{n}] = \bigoplus_{\lambda \in \Lambda^{\SO{n}}_+} V_{\lambda} \otimes [V_\lambda^*]^{\SO{n-k}}.
\]
The homogeneous space $\St{k}{n} = \SO{n}/\SO{n-k}$ is embedded in $\Mat_{k \times n} \simeq \bbC^k \otimes \bbC^n$, therefore the integral weights occurring in $\Mat_{k\times n}$ are the same as the integral weights occurring in the defining $\SO{n}$-representation $\bbC^n$. Since $\bbC^n = V_{(1)}$, the integral weights occurring in $\bbC^n$ are all the simple weights $\pm e_1 \vvirg \pm e_r$, where $n = 2r$ or $n = 2r+1$ depending on the parity. Denote by $\calC$ the convex hull of $\pm e_1 \vvirg \pm e_r$, that is, the cross-polytope in the weight space $\Lambda_\bbR = \Lambda \otimes_{\bbZ} \bbR$. 
By Lemma \ref{lemma: appearing weights}, we deduce
\begin{equation}\label{eq: schur weyl stiefel graded}
 \bbC[\St{k}{n}]_{\leq j} = \bigoplus_{\lambda \in j\mathcal C \cap \Lambda^{\SO{n}}_+} V_{\lambda} \otimes [V_\lambda^*]^{\SO{n-k}}.
 \end{equation}

In order to determine the dimensions of the direct summands, we introduce the formalism of Gelfand--Tsetlin polytopes.

\subsection{Gelfand--Tsetlin polytopes and invariants}
\begin{definition}
For $m \leq n$, define 
 \[
\calB(m,n) = \{ \lambda ^{\SO{i}} : i = m \vvirg n,\hspace{0.05 in} \lambda ^{\SO{i}} \text{ an integral dominant weight for $\SO{i}$}\}.
\]
The \emph{Bratteli poset} is the poset structure on $\calB(m,n)$ where $\lambda ^{\SO{i}} \preceq \mu^{\SO{j}}$ if and only if $i \leq j$ and $V_{\lambda^{\SO{i}}}$ appears in the decomposition of $V_{ \mu^{\SO{j}}}$ as a $\SO{i}$-representation.

This notion was introduced in \cite{Brat:InductiveLimits}. We refer to \cite{Durand:CombBratteli} for some information on the underlying combinatorial structure.
\end{definition}

\begin{lemma}\label{lemma: invariants from Brattelli}
 Let $\lambda$ be a dominant integral weight for $\SO{n}$. Let $m \leq n$. Then the dimension of the space of $\SO{m}$-invariants $\dim [V_\lambda]^{\SO{m}}$ equals the number of chains from $(0)^{\SO{m}}$ to $\lambda^{\SO{n}}$ in the Bratteli poset $\calB(m,n)$.
\end{lemma}
\begin{proof}
 This is a direct consequence of the branching rules described in Lemma \ref{lemma:branchingRules}. Indeed the restriction of an irreducible representation from $\SO{n}$ to $\SO{n-1}$ is multiplicity free, implying that every chain from $(0)^{\SO{m}}$ to $\lambda^{\SO{n}}$ gives a unique invariant and all these invariants are linearly independent.
\end{proof}

A useful combinatorial picture for recording chains in the Bratteli poset $\calB(m,n)$ is a Gelfand--Tsetlin pattern of shape $(\SO{m},\SO{n})$. This is a diagram of boxes placed in $n-m+1$ rows, indexed by integers $m \vvirg n$. The number of boxes in the $i$-th row equals the rank of $\SO{i}$ and the left border of the diagram is an overlapping descending staircase. The boxes are labeled by the integer coefficients of a dominant weight in terms of the simple weights and these labels interlace along each row according to the branching rules. More precisely, the labels have to satisfy the inequalities:

\setlength{\unitlength}{30pt}
\begin{align}\label{eq: GTinequalities1}
\begin{picture}(6,1.5)
\put(0,1){\ybox}
\put(1,1){\ybox}
\put(0.5,0){\ybox}
\put(-0.5,0){\ylabel{$\cdots$}}
\put(-1,1){\ylabel{$\cdots$}}
\put(2,1){\ylabel{$\cdots$}}
\put(1.5,0){\ylabel{$\cdots$}}
\put(0,1){\ylabel{\small $\mu_{i,j}$}}
\put(1,1){\ylabel{\small $\mu_{i+1,j}$}}
\put(0.5,0){\ylabel{\small $\mu_{i,j+1}$}}
\put(6,0.5){\ylabel{$ \mu_{i,j} \geq \mu_{i,j+1} \geq \mu_{i+1,j}$}}
\put(3,0.5){\ylabel{$\iff$}}
\end{picture} \\
\notag ~\\ 
\label{eq: GTinequalities2}
\begin{picture}(6,1.5)
\put(0,1){\ybox}
\put(0.5,0){\ybox}
\put(-0.5,0){\ylabel{$\cdots$}}
\put(-1,1){\ylabel{$\cdots$}}
\put(0,1){\ylabel{\small $\mu_{i,j}$}}
\put(0.5,0){\ylabel{\small $\mu_{i,j+1}$}}
\put(6,0.5){\ylabel{$ \mu_{i,j} \geq \mu_{i,j+1} \geq -\mu_{i,j}$}}
\put(3,0.5){\ylabel{$\iff$}}
\end{picture}\\
\notag ~ \\
\label{eq: GTinequalities3}
\begin{picture}(6,1.5)
\put(0,1){\ybox}
\put(1,1){\ybox}
\put(0.5,0){\ybox}
\put(-0.5,0){\ylabel{$\cdots$}}
\put(-1,1){\ylabel{$\cdots$}}
\put(0,1){\ylabel{\small $\mu_{i,j}$}}
\put(1,1){\ylabel{\small $\mu_{i+1,j}$}}
\put(0.5,0){\ylabel{\small $\mu_{i,j+1}$}}
\put(6,0.5){\ylabel{$ \mu_{i,j} \geq \mu_{i,j+1} \geq |\mu_{i+1,j}|$}}
\put(3,0.5){\ylabel{$\iff$}}
\end{picture}
\end{align}

These inequalities ensure that a filling of the Gelfand--Tsetlin pattern corresponds to a chain in the Bratteli poset. Conversely, any chain in the Bratteli poset will correspond to a filling.  

In Figure \ref{fig:47GTpattern}, we give an example of a Gelfand--Tsetlin pattern of shape $(\SO{3}, \SO{7})$. Notice that the zero in the row corresponding to $\SO{4}$ is forced by the third inequality in \eqref{eq: GTinequalities1}--\eqref{eq: GTinequalities3}.

\setlength{\unitlength}{15pt}
\begin{figure}[!htbp]
\begin{center}
\begin{picture}(4,6)
\put(0,4){\ylabel{6}}
\put(1,4){\ylabel{2}}
\put(2,4){\ylabel{2}}
\put(0.5,3){\ylabel{5}}
\put(1.5,3){\ylabel{2}}
\put(2.5,3){\ylabel{-1}}
\put(1,2){\ylabel{5}}
\put(2,2){\ylabel{1}}
\put(1.5,1){\ylabel{4}}
\put(2.5,1){\ylabel{0}}
\put(2,0){\ylabel{0}}

\put(5,4){\ylabel{$\SO{7}$}}
\put(5,3){\ylabel{$\SO{6}$}}
\put(5,2){\ylabel{$\SO{5}$}}
\put(5,1){\ylabel{$\SO{4}$}}
\put(5,0){\ylabel{$\SO{3}$}}

\put(0,4){\ybox}
\put(1,4){\ybox}
\put(2,4){\ybox}
\put(0.5,3){\ybox}
\put(1.5,3){\ybox}
\put(2.5,3){\ybox}
\put(1,2){\ybox}
\put(2,2){\ybox}
\put(1.5,1){\ybox}
\put(2.5,1){\ybox}
\put(2,0){\ybox}
\end{picture}
\caption{A chain from $\lambda^{\SO{7}}=(6,2,2)$ to $(0)^{\SO{3}}$ in the Bratteli poset $\mathcal B(3,7)$ given by a Gelfand--Tsetlin pattern of shape $(\SO{3},\SO{7})$} \label{fig:47GTpattern}
\end{center}
\end{figure}

In general, the shape of a Gelfand--Tsetlin diagram depends on the parity of $n$ and $m$ because the row corresponding to $\SO{i}$ has $\lfloor{\frac{i}{2}}\rfloor$ boxes. For reference, in Figure \ref{fig:chainExample1}, we give the shape when $n = 2r+1$ and $m = 2r' -1$ are both odd, from the weight $(0)$ for $\SO{m}$ to the weight $\lambda$ for $\SO{n}$. 

\setlength{\unitlength}{42pt}
\begin{figure}[!htbp]
\begin{center}
\begin{displaymath}
\scalebox{.7}{
\begin{picture}(10,8)
\put(0,6){\ylabel{$\lambda_1$}}
\put(1,6){\ylabel{$\lambda_2$}}
\put(2,6){\ylabel{$\lambda_3$}}
\put(3,6){\ylabel{$\cdots$}}
\put(4,6){\ylabel{$\cdots$}}
\put(5,6){\ylabel{$\lambda_{r-1}$}}
\put(6,6){\ylabel{$\lambda_r$}}
\put(0.5,5){\ylabel{$\mu_{1,1}$}}
\put(1.5,5){\ylabel{$\mu_{1,2}$}}
\put(2.5,5){\ylabel{$\mu_{1,3}$}}
\put(3.5,5){\ylabel{$\cdots$}}
\put(4.5,5){\ylabel{$\cdots$}}
\put(5.5,5){\ylabel{$\mu_{1,r-1}$}}
\put(6.5,5){\ylabel{$\mu_{1,r}$}}
\put(1,4){\ylabel{$\mu_{2,1}$}}
\put(2,4){\ylabel{$\mu_{2,2}$}}
\put(3,4){\ylabel{$\cdots$}}
\put(4,4){\ylabel{$\cdots$}}
\put(5,4){\ylabel{$\mu_{2,r-2}$}}
\put(6,4){\ylabel{$\mu_{2,r-1}$}}
\put(2,3){\ylabel{$\ddots$}}
\put(6,3){\ylabel{$\vdots$}}

\put(2,2){\ylabel{{\small $\mu_{k-2,1}$}}}
\put(3,2){\ylabel{$\mu_{k-2,2}$}}
\put(4,2){\ylabel{$\cdots$}}
\put(5,2){\ylabel{{\small $\mu_{k-2, r'-1}$}}}
\put(6,2){\ylabel{{\small $\mu_{k-2,r'}$}}}

\put(2.5,1){\ylabel{{\small $\mu_{k-1,1}$}}}
\put(3.5,1){\ylabel{$\mu_{k-1,2}$}}
\put(4.5,1){\ylabel{$\cdots$}}
\put(5.5,1){\ylabel{{\small $\mu_{k-1, r'-1}$}}}
\put(6.5,1){\ylabel{{\small $\mu_{k-1,r'}$}}}

\put(3,0){\ylabel{$0$}}
\put(4,0){\ylabel{$\cdots$}}
\put(5,0){\ylabel{$0$}}
\put(6,0){\ylabel{$0$}}

\put(0,6){\ybox}
\put(1,6){\ybox}
\put(2,6){\ybox}
\put(3,6){\ylabel{$\cdots$}}
\put(4,6){\ylabel{$\cdots$}}
\put(5,6){\ybox}
\put(6,6){\ybox}
\put(0.5,5){\ybox}
\put(1.5,5){\ybox}
\put(2.5,5){\ybox}
\put(3.5,5){\ylabel{$\cdots$}}
\put(4.5,5){\ylabel{$\cdots$}}
\put(5.5,5){\ybox}
\put(6.5,5){\ybox}
\put(1,4){\ybox}
\put(2,4){\ybox}
\put(3,4){\ylabel{$\cdots$}}
\put(4,4){\ylabel{$\cdots$}}
\put(5,4){\ybox}
\put(6,4){\ybox}
\put(2,3){\ylabel{$\ddots$}}
\put(6,3){\ylabel{$\vdots$}}

\put(2,2){\ybox}
\put(3,2){\ybox}
\put(4,2){\ylabel{$\cdots$}}
\put(5,2){\ybox}
\put(6,2){\ybox}

\put(2.5,1){\ybox}
\put(3.5,1){\ybox}
\put(4.5,1){\ylabel{$\cdots$}}
\put(5.5,1){\ybox}
\put(6.5,1){\ybox}

\put(3,0){\ybox}
\put(4,0){\ylabel{$\cdots$}}
\put(5,0){\ybox}
\put(6,0){\ybox}

\put(8,6){\ylabels{$\SO{2r+1}$}}
\put(8,5){\ylabels{$\SO{2r}$}}
\put(8,4){\ylabels{$\SO{2r-1}$}}
\put(8,3){\ylabel{$\vdots$}}
\put(8,2){\ylabels{$\SO{2r'+1}$}}
\put(8,1){\ylabels{$\SO{2r'}$}}
\put(8,0){\ylabels{$\SO{2r'-1}$}}
\end{picture}
}
\end{displaymath}
\caption{Gelfand--Tsetlin pattern of shape $(\SO{m},\SO{n})$ with $n,m$ both odd}
\label{fig:chainExample1}
\end{center}
\end{figure}
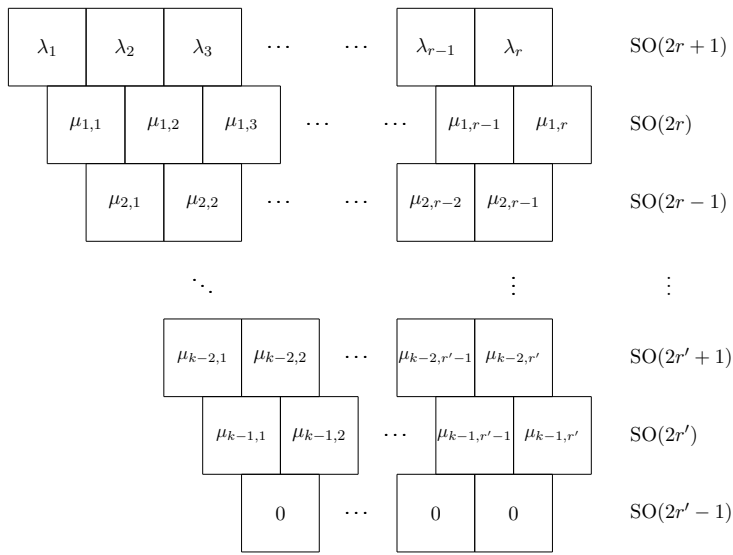

\begin{definition}
Let $n = 2r$ or $n = 2r+1$ depending on the parity. Let $\lambda \in \bbR^r$ be an $r$-tuple $\lambda = (\lambda_1 \vvirg \lambda_r)$ with $\lambda_1 \geq \cdots \geq \lambda_{r} \geq 0$ if $n$ is odd and $\lambda_1 \geq \cdots \geq \lambda_{r-1} \geq |\lambda_r|$ if $n$ is even. The {Gelfand--Tsetlin polytope} $GT^{\SO{n}}_{\SO{m}}(\lambda)$ is the set of all fillings of the Gelfand--Tsetlin pattern of shape $(\SO{m},\SO{n})$ with $\lambda$ in the top row, $(0)$ in the bottom row and filled by real numbers subject to the inequalities of \eqref{eq: GTinequalities1}--\eqref{eq: GTinequalities3}. 
\end{definition}

When $\lambda$ is a dominant integral weight for $\SO{n}$, then the integral points of the Gelfand--Tsetlin polytope $GT^{\SO{n}}_{\SO{m}}(\lambda)$ correspond to chains in the Bratteli poset and therefore via Lemma \ref{lemma: invariants from Brattelli} to $\SO{m}$-invariants in the $\SO{n}$-representation $V_\lambda$.

We establish the dimension of these polytopes in the range of interest.
\begin{lemma}\label{lemma: dimension GT polys}
 Fix $n \leq 2k-1$ with $n = 2r$ or $n =2r+1$ depending on the parity. Let $\lambda \in \bbR^{r}$ have distinct coefficients and let $GT^{\SO{n}}_{\SO{n-k}}(\lambda)$ be the corresponding Gelfand--Tsetlin polytope. Then
 \begin{align*}
 \dim GT^{\SO{n}}_{\SO{n-k}}(\lambda) = \left\{ \begin{array}{ll}
                                                 r(2k-r) - \binom{k+1}{2} & \text{if $n = 2r$ is even}; \\
                                                 r(2k-r-1) - \binom{k}{2} & \text{if $n = 2r+1$ is odd}.
                                                \end{array}\right.
 \end{align*}
\end{lemma}
\begin{proof}
The dimension of the Gelfand--Tsetlin polytope equals the number of labels of the Gelfand--Tsetlin pattern which are not forced to be $0$ by the inequalities \eqref{eq: GTinequalities1}--\eqref{eq: GTinequalities3}. Write $m = n-k$. Since $n \leq 2k-1$, we have $2m+1 \leq n$.

Suppose $m$ is odd, so $m +1$ is even and the row labeled $m$ of the Gelfand--Tsetlin pattern has $\frac{m+1}{2}$ boxes. From Figure \ref{fig:chainExample1}, observe that all but the first label in the second row from the bottom are forced to be $0$; similarly, all but the leftmost $i$ labels in the $(i+1)$-th row from the bottom are forced to be $0$ for $i = 1 \vvirg m$. This gives $1 + 2 + \cdots + m = \binom{m+1}{2}$ nonzero labels in the bottom $m + 1$ rows of the Gelfand--Tsetlin pattern: indeed, observe that the $(m+1)$-th row from the bottom corresponds to $\SO{2m}$ and all its labels are nonzero. Now consider the rows from $2m$ to $n$: the last row is fixed and its labels do not contribute to the dimension; the remaining $n - 2m -1$ rows contribute with a total of $2 \left[ \binom{r}{2} - \binom{m+1}{2} \right] + m$ labels if $n = 2r$ is even and $2 \left[ \binom{r}{2} - \binom{m+1}{2} \right] + m + r-1$ if $n = 2r+1$ is odd. Expanding the binomial coefficients, we obtain the result.
If $m$ is even, the calculation is similar.
\end{proof}

We point out that a result similar to Lemma \ref{lemma: dimension GT polys} holds in the range $n \geq 2k$, that is, when $\deg(\St{k}{n})$ equals the B\'ezout bound. However, in this case, the inequalities are more complicated and the statement is more involved. Although in principle one can compute $\deg(\St{k}{n})$ using this approach in the B\'ezout range, we prefer the geometric argument of Section \ref{section: bezout range} and do not provide additional details on the representation theoretic approach in these cases. 

We now characterize the degree of $\St{k}{n}$ in terms of volumes of Gelfand--Tsetlin polytopes, where \emph{volume} means the Euclidean volume in the real dimensional space given by Lemma \ref{lemma: dimension GT polys}.
\begin{theorem}\label{thm: degree as integral}
 Fix $k, n$ with $n \leq 2k-1$. Then
 \[
\deg(\St{k}{n}) = N! \int_{\calC \cap W} \vol\left(GT^{\SO{n}}_{\SO{n-k}}(\lambda)\right)\cdot \vol\left(GT^{\SO{n}}_{\SO{1}} (\lambda)\right) d\lambda,
 \]
 where $N = \dim \St{k}{n} = \binom{n}{2} - \binom{n-k}{2}$.
\end{theorem}
\begin{proof}
From equation \eqref{eq: schur weyl stiefel graded}, via Lemma \ref{lemma: invariants from Brattelli}, 
\begin{equation}\label{eq: deg stiefel as summation}
 \deg(\St{k}{n}) = N! \cdot \lim_{j \to \infty} \frac{1}{j^N} \sum_{\lambda \in j\mathcal C \cap \Lambda^{\SO{n}}_+} \left(\dim V_\lambda\right) \cdot \left(\dim [V_\lambda]^{\SO{n-k}}\right).
\end{equation}
Now, $\dim [V_\lambda]^{\SO{n-k}}$ equals the number of lattice points in $GT^{\SO{n}}_{\SO{n-k}}(\lambda)$. Similarly, $\dim V_\lambda$ is the number of invariants for the trivial group $\SO{1} \subseteq \SO{n}$, therefore it equals the number of lattice points in $GT^{\SO{n}}_{\SO{1}}(\lambda)$.

Using Lemma \ref{lemma: dimension GT polys}, whenever $\lambda$ has distinct coefficients, we obtain 
\[
 N - \left[ \dim GT^{\SO{n}}_{\SO{1}}(\lambda) + \dim GT^{\SO{n}}_{\SO{n-k}}(\lambda) \right] = r.
\]
This allows us to rewrite \eqref{eq: deg stiefel as summation} as
\begin{align*}
 \deg(\St{k}{n}) = N! \cdot \lim_{j \to \infty}  \sum_{\lambda \in j\mathcal C \cap \Lambda^{\SO{n}}_+} \frac{\dim V_\lambda}{ j^{\dim GT^{\SO{n}}_{\SO{1}}(\lambda)} } \cdot \frac{\dim [V_\lambda]^{\SO{n-k}}}{j^{\dim GT^{\SO{n}}_{\SO{n-k}}(\lambda)}} \\
 = N! \cdot \lim_{j \to \infty} \sum_{\substack{ \lambda \in \mathcal C \cap \frac{1}{j}\Lambda^{\SO{n}}_+}} \frac{\dim V_{j\lambda}}{j^{\dim GT^{\SO{n}}_{\SO{1}}(\lambda)} } \cdot \frac{ \dim [V_{j\lambda}]^{\SO{n-k}}}{j^{\dim GT^{\SO{n}}_{\SO{n-k}}(\lambda)}}.
\end{align*}
As $j \to \infty$ this summation converges to an integral and the number of rescaled lattice points converges to the volume of the Gelfand--Tsetlin polytope. We conclude
\[
 \deg(\St{k}{n}) = N! \int_{\calC \cap W} \vol \left( GT^{\SO{n}}_{\SO{1}}(\lambda)  \right)\cdot \vol\left( GT^{\SO{n}}_{\SO{n-k}}(\lambda)\right) d\lambda.
\]
\end{proof}

The volumes of the Gelfand--Tsetlin polytopes can be computed via straightforward integrals, using their definitions via the inequalities \eqref{eq: GTinequalities1}--\eqref{eq: GTinequalities3} which explicitly determine the range of each variable:
\begin{align*}
\text{inequality \eqref{eq: GTinequalities1}} &\longleftrightarrow \int_{\mu_{i+1,j}}^{\mu_{i,j}} 1 d\mu_{i,j+1} ,\\ 
\text{inequality \eqref{eq: GTinequalities2}} &\longleftrightarrow \int_{-\mu_{i,j}}^{\mu_{i,j}} 1d\mu_{i,j+1}, \\
\text{inequality \eqref{eq: GTinequalities3}} &\longleftrightarrow \int_{|\mu_{i+1,j}|}^{\mu_{i,j}} 1d\mu_{i,j+1}. \\
\end{align*}
In fact, we perform an additional reduction: for the integral associated to \eqref{eq: GTinequalities2}, we have $\int_{-\mu_{i,j}}^{\mu_{i,j}} 1d\mu_{i,j+1} = 2\int_{0}^{\mu_{i,j+1}} 1 d\mu_{i,j+1}$. This allows us to assume that the rightmost label of every row of the Gelfand--Tsetlin pattern is nonnegative and simplifies the integral associated to \eqref{eq: GTinequalities3} as well, providing $ \int_{|\mu_{i+1,j}|}^{\mu_{i,j}} 1d\mu_{i,j+1} =  \int_{\mu_{i+1,j}}^{\mu_{i,j}} 1d\mu_{i,j+1}$.

After this simplification, the volume of the Gelfand--Tsetlin polytope is provided by a series of nested integrals, where one counts twice every integral whose integration variable is the label of the rightmost box of a row corresponding to $\SO{i}$ with $i$ even.

\begin{example}\label{ex:47}
Consider the general pattern of shape $(\SO{3},\SO{7})$:
\setlength{\unitlength}{25pt}
\begin{displaymath}
\begin{picture}(4,6)
\put(0,4){\ylabel{$\lambda_1$}}
\put(1,4){\ylabel{$\lambda_2$}}
\put(2,4){\ylabel{$\lambda_3$}}
\put(0.5,3){\ylabel{$\mu_{1,1}$}}
\put(1.5,3){\ylabel{$\mu_{1,2}$}}
\put(2.5,3){\ylabel{$\mu_{1,3}$}}
\put(1,2){\ylabel{$\mu_{2,1}$}}
\put(2,2){\ylabel{$\mu_{2,2}$}}
\put(1.5,1){\ylabel{$\mu_{3,1}$}}
\put(2.5,1){\ylabel{$\mu_{3,2}$}}
\put(2,0){\ylabel{0}}

\put(5,4){\ylabel{$\SO{7}$}}
\put(5,3){\ylabel{$\SO{6}$}}
\put(5,2){\ylabel{$\SO{5}$}}
\put(5,1){\ylabel{$\SO{4}$}}
\put(5,0){\ylabel{$\SO{3}$}}

\put(0,4){\ybox}
\put(1,4){\ybox}
\put(2,4){\ybox}
\put(0.5,3){\ybox}
\put(1.5,3){\ybox}
\put(2.5,3){\ybox}
\put(1,2){\ybox}
\put(2,2){\ybox}
\put(1.5,1){\ybox}
\put(2.5,1){\ybox}
\put(2,0){\ybox}

\end{picture}
\end{displaymath}

Note that the inequality (\ref{eq: GTinequalities3}) implies that $\mu_{3,2}=0$ and so there are only $6$ free variables. The volume of $GT^{\SO{7}}_{\SO{3}}(\lambda)$ is given by
\[
\int_{\lambda_2}^{\lambda_1}
\int_{\lambda_3}^{\lambda_2}
2\int_{0}^{\lambda_3}
\int_{\mu_{1,2}}^{\mu_{1,1}}
\int_{\mu_{1,3}}^{\mu_{1,2}}
\int_{\mu_{2,2}}^{\mu_{2,1}}
1
d\mu_{3,1}
d\mu_{2,2}
d\mu_{2,1}
d\mu_{1,3}
d\mu_{1,2}
d\mu_{1,1}
\]
which evaluates to
\[
\vol\left(GT^{\SO{7}}_{\SO{3}}(\lambda)\right)=\frac 1 6 (\lambda_1-\lambda_2)(\lambda_2-\lambda_3)(\lambda_1-\lambda_3)\lambda_1\lambda_2\lambda_3. 
\]
In particular, note the factor of $2$ arising in the integration with respect to $\mu_{1,3}$ between $0$ and $\lambda_3$.
We point out that this volume is an alternating function in the $\lambda_j$'s, evident from the outermost two integrals. Moreover, it is divisible by $\lambda_3$ (and thus $\lambda_1$ and $\lambda_2$ by the alternating property) evident from the third outermost integral.
\end{example}

\subsection{Alternating functions and volumes of Gelfand--Tsetlin polytopes}

In this section, we use an induction argument to determine the volumes of the Gelfand--Tsetlin polytopes relevant to the calculation of $\deg \St{k}{n}$.

In Example \ref{ex:47}, we saw that the volume of $GT^{\SO{7}}_{\SO{3}}(\lambda)$ is an alternating polynomial in $\lambda$. It is clear that this is a general fact, because of the last sequence of integrals in $\vol \left(GT^{\SO{7}}_{\SO{3}}(\lambda)\right)$.

We record some facts about alternating polynomials referring to \cite[Ch. I]{Macdon:SymmetricFunctions}. Given an integer partition $\mu = (\mu_1 \vvirg \mu_r)$, define the alternating polynomial
\[
a_\mu (\lambda_1 \vvirg \lambda_r) = \det \left[ \lambda_j^{\mu_i + r - i} \right] .
\]
We remark that our notation differs from the usual notation which uses the subscript $\mu + (r-1 \vvirg 1,0)$ instead of $\mu$ for the alternating polynomial $a_\mu$.

We record two useful results on integration of alternating functions. The first gives the result of the integral of a product of two alternating functions on the standard simplex.
\begin{lemma}
 \label{lemma: integral product alternating}
Let $\Delta_r$ be the convex hull of the origin and the standard $r-1$-simplex in $\bbR^r$. Let $\mu=(\mu_1,\ldots,\mu_r)$ and $\nu=(\nu_1,\ldots,\nu_r)$ be two partitions. Then 
\[
\int_{\Delta_r} a_\mu(\lambda)a_\nu(\lambda) d\lambda = \frac{r!}{(r^2+|\mu|+|\nu|)!} \det\left(\begin{bmatrix}
(\nu_i+\mu_j +2r- i-j)!
\end{bmatrix}_{i,j=1}^r\right) 
\]
where $|\mu| = \sum \mu_i$ and $|\nu| = \sum \nu_i$.
\end{lemma}
\begin{proof}
The proof is an explicit calculation obtained by expanding the determinants defining $a_\mu(\lambda)$ and $a_\nu(\lambda)$. Given a permutation $\sigma$, write $(-1)^\sigma$ for its sign.
\begin{align*}
\int_{\Delta_r}a_\mu(\lambda)a_\nu(\lambda)d\lambda & = \int_{\Delta_r} \sum_{\sigma,\tau \in \mathfrak{S}_r} (-1)^{\sigma \circ \tau}  \prod_{i=1}^r \lambda_i^{\mu_{\sigma(i)} + \nu_{\tau(i)} + 2r - \sigma(i)-\tau(i)}d\lambda \\
 &=\sum_{\sigma,\tau \in \mathfrak{S}_r} (-1)^{\sigma\circ\tau} \left(\int_{\Delta_r} \prod_{i=1}^r\lambda_i^{\mu_{\sigma(i)}+\nu_{\tau(i)}+2r -\sigma(i)-\tau(i)} d \lambda\right).
 \end{align*}
 The integral of a monomial over a simplex is given by \cite[Lemma 4.23]{Milne:AlgebraicNumberTheory}. Applying this to our expression gives
 \begin{align*}
  & \sum_{\sigma,\tau \in \mathfrak{S}_r} (-1)^{\sigma\circ\tau} \frac{\prod_{i=1}^r(\mu_{\sigma(i)}+\nu_{\tau(i)}+2r-\sigma(i)-\tau(i))!}{(r+\left(\sum_{i=1}^r 2r \right)+ \mu_{\sigma(i)}+\nu_{\tau(i)}-\sigma(i)-\tau(i))!}\\
 =&\frac{1}{(r^2+|\mu|+|\nu|)!}  \sum_{\sigma,\tau \in \mathfrak{S}_r} (-1)^{\sigma\circ\tau} \prod_{i=1}^r(\mu_{\sigma(i)}+\nu_{\tau(i)}+2r-\sigma(i)-\tau(i))! 
 \end{align*}
 \begin{align*}
 =&\frac{r!}{(r^2+|\mu|+|\nu|)!}\sum_{\tau \in \mathfrak{S}_r}(-1)^{\tau}  \prod_{i=1}^r(\mu_{i}+\nu_{\tau(i)}+2r-i-\tau(i))! \\
=&\frac{r!}{(r^2+|\mu|+|\nu|)!} \det\left([(\mu_i+\nu_j+2r-i-j)!]\right)_{i,j=1}^r.
\end{align*}
\end{proof}

The second result provides a formula for the integral of alternating functions in terms of the integration bounds.
\begin{lemma}
\label{lemma: integral alternating}
Let $\pi$ be a partition $\pi = (\pi_1 \vvirg \pi_r)$. Then
\begin{align*}
 \int_{\lambda_2}^{\lambda_1} \cdots \int_{\lambda_{r+1}}^{\lambda_{r}} a_\pi (\mu_1 \vvirg \mu_r) d\mu_r \cdots d\mu_1 &= \frac{1}{\prod_1^r(\pi_j + r - j+1)} \cdot a_{(\pi ,0)} (\lambda_1 \vvirg \lambda_{r+1}). \\
\end{align*}
\end{lemma}
\begin{proof}
Consider the determinant representation of $a_\pi(\mu)$ and notice that each variable appears only in a single column of the corresponding matrix. By linearity, this implies that the integration can be performed directly on the entries of the matrix:

\begin{align*}
 \int_{\lambda_2}^{\lambda_1} \cdots &\int_{\lambda_{r+1}}^{\lambda_{r}} a_\pi (\mu_1 \vvirg \mu_r) d\mu_r \cdots d\mu_1  \\
 &=\det \left[ \begin{array}{ccc} \int_{\lambda_2}^{\lambda_1} \mu_1^{\pi_1 + r-1} d\mu_1 & \cdots & \int_{\lambda_2}^{\lambda_1} \mu_r^{\pi_1+r-1} d\mu_r  \\ 
 \vdots & &  \vdots \\
 \int_{\lambda_2}^{\lambda_1} \mu_1^{\pi_r } d\mu_1  &\cdots & \int_{\lambda_2}^{\lambda_1} \mu_r^{\pi_r} d\mu_r   \end{array}
 \right]  \\
&= \frac{1}{\prod (\pi_j + r-j+1)} \det \left[ \begin{array}{ccc}
\lambda_1^{\pi_1 + r} -\lambda_2^{\pi_1 + r} &\cdots &  \lambda_{r}^{\pi_1 + r} - \lambda_{r+1}^{\pi_1 + r} \\ 
\vdots & &  \vdots \\
\lambda_1^{\pi_r+1} -\lambda_2^{\pi_r+1} &\cdots &  \lambda_{r}^{\pi_r+1} -\lambda_{r+1}^{\pi_r+1} \\ 
 \end{array}
 \right]  \\
&= \frac{1}{\prod (\pi_j + r-j+1)} \det \left[ \begin{array}{cccc}
\lambda_1^{\pi_1+ r } - \lambda_2^{\pi_1 + r} & \cdots & \lambda_{r}^{\pi_1 + r} - \lambda_{r+1}^{\pi_1 + r} & \lambda_{r+1}^{\pi_1 + r} \\ 
\vdots & &  \vdots & \vdots \\
\lambda_1^{\pi_r+1} - \lambda_2^{\pi_r+1} & \cdots & \lambda_{r}^{\pi_r+1} -\lambda_{r+1}^{\pi_r+1} & \lambda_{r+1}^{\pi_r+1} \\ 
0 &  \cdots & 0 & 1\end{array}
 \right]  \\
&= \frac{1}{\prod (\pi_j + r-j+1)} \det \left[ \begin{array}{cccc}
\lambda_1^{\pi_1+ r } & \lambda_2^{\pi_1 + r} & \cdots & \lambda_{r+1}^{\pi_1 + r}  \\ 
\vdots & \vdots & &  \vdots \\
\lambda_1^{\pi_r+1} &\lambda_2^{\pi_r+1} &\cdots & \lambda_{r+1}^{\pi_r+1}  \\ 
1 & 1 & \cdots & 1\end{array}
 \right]  \\ &=\frac{1}{\prod (\pi_j + r-j+1)} \cdot a_{(\pi,0)}(\lambda_1 \vvirg \lambda_{r+1}).
\end{align*}
\end{proof}

Define recursively the following partitions. Let $\Omega_{k , 2k-1} = \underbrace{(1 \vvirg 1)}_{k-1}$ and let 
\begin{equation}\label{def: omega}
 \Omega_{k,n} = \left\{ \begin{array}{ll} 
                         (\Omega_{k-1,n-1},0) & \text{if $n$ is even} \\
                         \Omega_{k-1,n-1} + (1 \vvirg 1) & \text{if $n$ is odd}.
                        \end{array}\right.
\end{equation}
A closed expression for $\Omega_{k,n}$ can be obtained by induction and it is given by
\begin{equation*}\label{def: omega closed}
 \Omega_{k,n} = \left\{  \begin{array}{ll} 
                         (\underbrace{k-r \vvirg k-r}_{n-k}, k-r-1 \vvirg 0) & \text{if $n = 2r$ is even} \\
                         (\underbrace{k-r \vvirg k-r}_{n-k}, k-r-1 \vvirg 1) & \text{if $n=2r+1$ is odd}.
                         \end{array}\right.
\end{equation*}

Notice that the recursion reaches all pairs $(k,n)$ with $n \leq 2k-1$. For reference, Table \ref{table: omega table} contains the first values of $\Omega_{k,n}$.

\begin{table}[!htpb]
\begin{center}
\begin{tabular}{|l|c|c|c|c|c|c|c|c|c|c|}
\hline
$k \backslash n$ & 1 & 2 & 3 & 4 & 5 & 6 &7 &8&9&10\\
\hline
1&\cellcolor{green} &\cellcolor{green!21}(0)&\cellcolor{green!21}&\cellcolor{green!21}&\cellcolor{green!21}& \cellcolor{green!21}&\cellcolor{green!21}&\cellcolor{green!21}&\cellcolor{green!21}&\cellcolor{green!21}
\\
\hline
2&*&\cellcolor{blue!50}&\cellcolor{green}(1)&\cellcolor{green!21}&\cellcolor{green!21}&\cellcolor{green!21}&\cellcolor{green!21}&\cellcolor{green!21}&\cellcolor{green!21}&\cellcolor{green!21}
\\
\hline
3&*&*&\cellcolor{blue!50}&(1,0)\cellcolor{blue!25}&\cellcolor{green}(1,1)&\cellcolor{green!21}&\cellcolor{green!21}&\cellcolor{green!21}&\cellcolor{green!21}&\cellcolor{green!21}
\\
\hline
4&*&*&*&\cellcolor{blue!50}&(2,1)\cellcolor{blue!25}&(1,1,0)\cellcolor{blue!25}&\cellcolor{green}(1,1,1)&\cellcolor{green!21}&\cellcolor{green!21}&\cellcolor{green!21}
\\
\hline
5&*&*&*&*&\cellcolor{blue!50}&(2,1,0)\cellcolor{blue!25}&(2,2,1)\cellcolor{blue!25}&(1,1,1,0)\cellcolor{blue!25}&(1,1,1,1)\cellcolor{green}&\cellcolor{green!21}
\\
\hline
6&*&*&*&*&*&\cellcolor{blue!50}&(3,2,1)\cellcolor{blue!25}&(2,2,1,0)\cellcolor{blue!25}&(2,2,2,1)\cellcolor{blue!25}&(1,1,1,1,0)\cellcolor{blue!25}
\\
\hline
7&*&*&*&*&*&*&\cellcolor{blue!50}&\cellcolor{blue!25}(3,2,1,0) &(3,3,2,1)\cellcolor{blue!25}&(2,2,2,1,0)\cellcolor{blue!25}
\\
\hline
8&*&*&*&*&*&*&*&\cellcolor{blue!50} &\cellcolor{blue!25}(4,3,2,1)&(3,3,2,1,0)\cellcolor{blue!25}
\\
\hline
9&*&*&*&*&*&*&*&*&\cellcolor{blue!50}&\cellcolor{blue!25}(4,3,2,1,0) 
\\
\hline
10&*&*&*&*&*&*&*&*&*&\cellcolor{blue!50}
\\
\hline
\end{tabular}
\caption{Partitions $\Omega_{k,n}$ from \eqref{def: omega}. The bases of the recursion are the dark green boxes; the recursive steps move south east.}\label{table: omega table}
\end{center}
\end{table}

\begin{proposition}\label{prop: volume GT 2k k}
The volume of $GT^{\SO{2k-1}}_{\SO{k-1}}(\lambda)$ is 
\[
\vol\left(GT^{\SO{2k-1}}_{\SO{k-1}}(\lambda)\right) = \frac{2}{\prod_{j=1}^{k-1}j!}a_{\Omega_{k,2k-1}}(\lambda) .
\]
\end{proposition}
\begin{proof}

Let $n=2k-1$. As in the proof of Lemma \ref{lemma: dimension GT polys}, observe that only some of the labels on the Gelfand--Tsetlin pattern can be nonzero: only the $i$ leftmost labels in the row corresponding to $\SO{k-1+i}$ are nonzero, for $i = 1 \vvirg 2k-2$. In particular, the row corresponding to $\SO{2k-2}$ has no labels identically equal to $0$. This shows 
\begin{equation}\label{eq: vol GT 2k k}
\vol\left(GT^{\SO{2k-1}}_{\SO{k-1}}(\lambda)\right) =  \int_{\lambda_{\lambda_2}}^{\lambda_{1}} \cdots\int_{\lambda_{k-1}}^{\lambda_{k-2}} 2\int_{0}^{\lambda_{k-1}} \vol(T_{k-1}(\mu_1 \vvirg \mu_{k-1})) d\mu_{k-1} \cdots d\mu_1,
\end{equation}
where $T_{\ell}(\mu_1 \vvirg \mu_{\ell})$ is the polytope defined by the same inequalities as in \eqref{eq: GTinequalities1}--\eqref{eq: GTinequalities3} and the triangular shape

\setlength{\unitlength}{25pt}
\begin{displaymath}
\begin{picture}(6,5)
\put(0,4){\ylabel{$\mu_1$}}
\put(1,4){\ylabel{$\mu_2$}}
\put(3,4){\ylabel{$\cdots$}}
\put(5,4){\ylabel{$\mu_{\ell-1}$}}
\put(6,4){\ylabel{$\mu_{\ell}$}}
\put(1.5,2){\ylabel{$\ddots$}}
\put(4.5,2){\ylabel{$\iddots$}}

\put(0,4){\ybox}
\put(1,4){\ybox}
\put(5,4){\ybox}
\put(6,4){\ybox}
\put(0.5,3){\ybox}
\put(1.5,3){\ybox}
\put(4.5,3){\ybox}
\put(5.5,3){\ybox}
\put(2.5,1){\ybox}
\put(3.5,1){\ybox}
\put(3,0){\ybox}

\end{picture}
\end{displaymath}

Observe that $\dim T_{\ell}(\mu_1 \vvirg \mu_{\ell}) = \binom{\ell}{2}$ and its volume is an alternating polynomial in the variables $\mu_1 \vvirg \mu_{\ell}$. There is a unique, up to scale, alternating polynomial of degree $\binom{\ell}{2}$ in $\ell$ variables and it is the Vandermonde determinant. Therefore, 
\[
\vol (T_{\ell}(\mu_1 \vvirg \mu_{\ell}) ) = \kappa_\ell a_{(0 \vvirg 0)}(\mu). 
\]
for some constant $\kappa_\ell$. We use induction on $\ell$ to determine $\kappa_\ell = \frac{1}{\prod_{j=1}^{\ell-1} j!}$. This holds when $\ell = 2$.

For $\ell \geq 3$, notice 
\begin{align*}
\vol (T_{\ell}(\mu_1 \vvirg \mu_{\ell}) )  &= \int_{\mu_2}^{\mu_1} \cdots \int_{\mu_{\ell}}^{\mu_\ell} \vol (T_{\ell-1}(\nu_1 \vvirg \nu_{\ell-1}  ) ) d\nu_{\ell-1} \cdots d \nu_1 \\ 
&= \kappa_{\ell-1} \int_{\mu_2}^{\mu_1} \cdots \int_{\mu_{\ell}}^{\mu_{\ell-1}} a_{(0 \vvirg 0)} (\nu_1  \vvirg \nu_{\ell-1})  d\nu_{\ell-1} \cdots d \nu_1 \\
&= \frac{1}{\prod_{j=1}^{\ell-2} j!} \cdot \frac{1}{\prod_1^{\ell-1} (\ell-j)} a_{(0\vvirg 0,0)}(\mu_1 \vvirg \mu_{\ell}),
\end{align*}
where in the last line we used Lemma \ref{lemma: integral alternating}; since $\prod_1^{\ell-1} (\ell-j) = (\ell-1)!$, we obtain the desired value of $\kappa_\ell$.

It remains to evaluate the integral in \eqref{eq: vol GT 2k k}. From \eqref{eq: vol GT 2k k}, we see 
\[
\vol\left(GT^{\SO{2k-1}}_{\SO{k-1}}(\lambda)\right) = 2 \cdot \vol (T_{k} (\lambda_1 \vvirg \lambda_{k-1},0)).
\]
This concludes the proof because 
\begin{align*}
2 \cdot \vol (T_{k} (\lambda_1 \vvirg \lambda_{k-1},0))&=2\left(\frac{1}{\prod_1^{k-1}j!}a_{(0,\ldots,0)}(\lambda_1,\ldots,\lambda_{k-1},0)\right)\\
&=\frac{2}{\prod_1^{k-1}j!}\lambda_1\cdots\lambda_{k-1}a_{(0,\ldots,0)}(\lambda_1,\ldots,\lambda_{k-1}) \\
&= \frac{2}{\prod_1^{k-1}j!} a_{(1,\ldots,1)}(\lambda_1,\ldots,\lambda_{k-1}) \\
&= \frac{2}{\prod_1^{k-1}j!} a_{\Omega_{k,2k-1}}(\lambda_1,\ldots,\lambda_{k-1}) .
\end{align*}

\end{proof}

Proposition \ref{prop: volume GT 2k k} provides the base of the induction for the following result.
\begin{theorem}\label{thm: volume GT general}
 Let $n \leq 2k-1$ with $n = 2r$ or $n = 2r+1$ depending on its parity and let $\lambda =(\lambda_1 \vvirg \lambda_r)$. Then
 \begin{align*}
  \vol \left( GT^{\SO{n}}_{\SO{n-k}} (\lambda) \right)&= \frac{2^{k-r}}{\prod_{j=1}^{r} ((\Omega_{k,n})_j +r-j)!} \cdot a_{\Omega_{k,n}} (\lambda) \\
&=           \frac{2^{k-r}}{\prod_1^{n-k} (k-j)! \cdot \prod_{n-k+1}^r (n-2j)!} \cdot a_{ \Omega_{k,n}}(\lambda).
    \end{align*}
\end{theorem}
\begin{proof}
Since $n \leq 2k-1$, there exists a nonnegative integer $p$ such that $(k,n) = (\ell +p, 2\ell -1 + p)$. We use induction on $p$. Notice that $n-k =\ell-1$ does not depend on $p$. When $p = 0$, the statement is true by Proposition \ref{prop: volume GT 2k k}. 

Notice that $\vol \left(GT^{\SO{n}}_{\SO{n-k}}(\lambda) \right)$ is obtained by integrating $\vol \left(GT^{\SO{n-1}}_{\SO{n-k}}(\lambda)\right)$ in the labels of the second (from the top) row of the Gelfand--Tsetlin pattern, see Figure \ref{fig:chainExample1}.

We consider two cases depending on the parity of $p$. 

\underline{Let $p$ be odd}. In this case $n = 2\ell-1 + p$ is even and if $\SO{n}$ has rank $r$ then $\SO{n-1}$ has rank $r-1$. We have
\begin{align*}
 \vol \left(GT^{\SO{n}}_{\SO{\ell-1}}(\lambda)\right) = \int_{\lambda_2}^{\lambda_1} \cdots \int_{\lambda_{r}}^{\lambda_{r-1}}  \vol \left(GT^{\SO{n-1}}_{\SO{\ell-1}}(\mu_1 \vvirg \mu_{r-1}) \right) d\mu_{r-1} \cdots d\mu_{1} \\
=\frac{2^{k-1-(r-1)}}{\prod_{j=1}^{r-1} ( (\Omega_{k-1,n-1})_j +r-1-j)!} \int_{\lambda_2}^{\lambda_1} \cdots \int_{\lambda_{r}}^{\lambda_{r-1}}  a_{\Omega_{k-1,n-1}} (\mu) \mu_{r-1} \cdots d\mu_1 
\end{align*}
where we use the inductive hypothesis for $p-1$ to compute $\vol \left(GT^{\SO{n-1}}_{\SO{\ell-1}}(\mu_1 \vvirg \mu_{r-1}) \right)$.

Applying Lemma \ref{lemma: integral alternating}, we obtain
\begin{align*}
  \vol \left(GT^{\SO{n}}_{\SO{\ell-1}}(\lambda) \right)= &\frac{2^{k-r}}{\prod_{j=1}^{r-1} ((\Omega_{k-1,n-1})_j! + (r-1)-j) } \\ &\cdot \frac{1}{\prod_1^{r-1} ( (\Omega_{k-1,n-1})_j + (r-1)-j+1 } a_{(\Omega_{k-1,n-1} ,0)} (\lambda).
\end{align*}
Since $n$ is even, we have $(\Omega_{k-1,n-1} ,0) = \Omega_{k,n}$ so 
\[
\prod_{j=1}^{r-1} ((\Omega_{k-1,n-1})_j + (r-1)-j)! \cdot \prod_1^{r-1} ( (\Omega_{k-1,n-1})_j + r-j) = \prod_1^{r} (\Omega_{k,n})_j + r -j)!.
 \]
This concludes the proof when $p$ is odd.

\underline{Let $p$ be even}. In this case $n = 2\ell -1 + p$ is odd, so $\SO{n}$ and $\SO{n-1}$ have rank $r$. We have
\begin{align*}
 \vol \left(GT^{\SO{n}}_{\SO{\ell-1}})(\lambda)\right) = \int_{\lambda_2}^{\lambda_1} \cdots \int_{\lambda_{r}}^{\lambda_{r-1}} 2\int_{0}^{\lambda_r}  \vol \left(GT^{\SO{n-1}}_{\SO{\ell-1}}(\mu_1 \vvirg \mu_{r}) \right) d\mu_{r} \cdots d\mu_{1}  \\
=\frac{2^{k-1-r)}}{\prod_{j=1}^r ( (\Omega_{k-1,n-1})_j +r-j)!} \int_{\lambda_2}^{\lambda_1} \cdots \int_{\lambda_{r}}^{\lambda_{r-1}} 2\int_{0}^{\lambda_r}  a_{\Omega_{k-1,n-1}} (\mu) \mu_{r} \cdots d\mu_1 
\end{align*}
Similarly to the proof of Proposition \ref{prop: volume GT 2k k}, we regard the last integration bound $0$ as a variable $\lambda_{r+1}$ and then evaluate it to $0$. By Lemma \ref{lemma: integral alternating}, we deduce
\begin{align*}
\vol \left(GT^{\SO{n}}_{\SO{\ell-1}}(\lambda) \right)=&  \frac{2^{k-1-r}}{\prod_{j=1}^r ( (\Omega_{k-1,n-1})_j +r-j)!} \\ &\cdot  \frac{1}{ \prod_1^r ((\Omega_{k-1,n-1})_j + r +1 - j)} \cdot  2 \cdot a_{ ( \Omega_{k-1,n-1},0)} (\lambda, \lambda_{r+1}) | _{\lambda_{r+1} = 0}.
\end{align*}
From properties of alternating functions $a_{ (\pi , 0) }(\lambda_1 \vvirg \lambda_{r+1}) |_{\lambda_{r+1} =0} = a_{ (\pi + (1 \vvirg 1))}(\lambda_1 \vvirg \lambda_r)$ for every partition $\pi$. Since $n$ is odd, we have $\Omega_{k,n} = \Omega_{k-1,n-1} + (1 \vvirg 1)$. This allows us to conclude:
\begin{align*}
 \vol \left(GT^{\SO{n}}_{\SO{\ell-1}}(\lambda) \right) = &\frac{2 \cdot 2^{k-1 - r}}{ \prod_1^r ((\Omega_{k,n})_j + r - j -1)!} \cdot \frac{1}{ \prod_1^r ((\Omega_{k,n})_j + r - j)} a_{ ( \Omega_{k-1,n-1} + (1 \vvirg 1))} (\lambda) \\
 = &\frac{2^{k- r}}{ \prod_1^r ((\Omega_{k,n})_j + r - j)!}a_{ \Omega_{k,n} } (\lambda).
\end{align*}
This concludes the proof for even $p$.

The second equality in the statement of the theorem is obtained by writing $\Omega_{k,n}$ explicitly in the denominator.
\end{proof}

We record separately the instance of Theorem \ref{thm: volume GT general} when $k = n-1$; by Lemma \ref{lemma: invariants from Brattelli} and the discussion after that, these are the Gelfand--Tsetlin polytopes controlling the dimension of irreducible $\SO{n}$-representations. Indeed, when $\lambda$ is a dominant integral weight for $\SO{n}$, the volume of $GT^{\SO{n}}_{\SO{1}}(\lambda)$ can be recovered directly from Weyl dimension formula, see e.g. \cite{DerkKraft:ConstructiveInvTheory}. 

\begin{corollary}\label{corol: volume GT SOn}
Let $n$ be a positive integer. Then 
 \[
  \vol \left(GT^{\SO{n}}_{\SO{1}}(\lambda)\right) = \left\{ \begin{array}{ll} 
                                                  \frac{2^{r - 1}}{\prod_1^r (2(r-j))!} a_{(r-1,r-2 \vvirg 0)} (\lambda) & \text{if $n = 2r$ is even} \\
                                                  \frac{2^{r }}{\prod_1^r (2(r-j)+1)!} a_{(r,r-1 \vvirg 1)} (\lambda) & \text{if $n = 2r+1$ is odd} \\
                                                 \end{array}\right.
 \]
\end{corollary}

\subsection{Degrees of Stiefel manifolds via volumes of Gelfand--Tsetlin polytopes}
We have now completed all the preparatory work to determine the degree of $\St{k}{n}$ when $n \leq 2k-1$.

\begin{theorem}\label{thm: degformula}
 Let $n \leq 2k-1$. Then
 \[
  \deg \St{k}{n} = 2^k \det \begin{bmatrix}
  \begin{pmatrix}{(\Omega_{k,n})_i+(\Omega_{n-1,n})_j+2r-i-j}\\ {(\Omega_{k,n})_i+r-i}\end{pmatrix}
  \end{bmatrix}_{1 \leq i,j \leq r}
 \]
\end{theorem}
\begin{proof}

From Theorem \ref{thm: degree as integral} we have
 \[
\deg(\St{k}{n}) = N! \int_{\calC \cap W} \vol\left(GT^{\SO{n}}_{\SO{n-k}}(\lambda)\right)\cdot \vol\left(GT^{\SO{n}}_{\SO{1}} (\lambda)\right) d\lambda,
 \]
and from Theorem \ref{thm: volume GT general} we write
 \[
\deg(\St{k}{n}) = \frac{ N!2^{k+n-1-2r}}{\prod_{j=1}^r((\Omega_{k,n})_j+r-j)!\cdot ((\Omega_{n-1,n})_j+r-j)!}\int_{\calC \cap W} a_{\Omega_{k,n}}(\lambda) a_{\Omega_{n-1,n}}(\lambda) d\lambda.
 \]
\underline{When $n=2r$ is even},
 \[
\mathcal C \cap  W=\{(\lambda_1,\ldots,\lambda_r)|\lambda_1 \geq \lambda_2 \geq \cdots \geq \lambda_{r-1} \geq |\lambda_r|\}.  
 \]
 Since the integrand is alternating in $\lambda$, the integral over $\mathcal C \cap W$ is equal to $\frac{2}{r!}$ times the integral over $\Delta_r$ and so we may write
  \[
\deg(\St{k}{n}) =  \frac{N! \cdot 2^{k}}{r! \cdot \prod_{j=1}^r((\Omega_{k,n})_j+r-j)! \cdot ((\Omega_{n-1,n})_j+r-j)!}\int_{\Delta_r} a_{\Omega_{k,n}}(\lambda) a_{\Omega_{n-1,n}}(\lambda) d\lambda.
 \]
We compute this integral using Lemma \ref{lemma: integral product alternating}:
\[
 \int_{\Delta_r} a_{\Omega_{k,n}}(\lambda) a_{\Omega_{n-1,n}}(\lambda) d\lambda = \frac{r!}{(r^2 + | \Omega_{k,n}|+|\Omega_{n-1,n}| )!} \det M = \frac{r!}{N!} \det M
\]
where $M$ is the $r \times r$ matrix with $(i,j)$-th entry $M_{i,j}=((\Omega_{k,n})_i+(\Omega_{n-1,n})_j+2r-i-j)!$.

This yields
\[
\deg(\St{k}{n}) =2^k \cdot \frac{1}{\prod_{j=1}^r((\Omega_{k,n})_j+r-j)!((\Omega_{n-1,n})_j+r-j)!} \cdot \det M.
\]
Distributing the factor $1/((\Omega_{k,n})_j+r-j)!$ in the $j$-th column of the matrix and the factor $1/((\Omega_{n-1,n})_i+r-i)!$ in the $i$-th row provides the desired determinant when $n$ is even.

\underline{When $n=2r+1$ is odd}, the proof is essentially the same. The only difference is that
\[
\mathcal C \cap  W = \{(\lambda_1,\ldots,\lambda_r)| \lambda_1 \geq \lambda_2 \geq \cdots \geq \lambda_r \geq 0\} 
\]
therefore the integral over $\mathcal C \cap  W$ equals $\frac{1}{r!}$ times the integral over $\Delta_r$. Since in this case $2r+1 = n$, the power of $2$ simplifies to $2^k$ as was the case when $n$ was odd.
\end{proof}

\subsection{Non-intersecting lattice path interpretation}
As is the case of the formula for $\deg(\SO{n})$ in \cite{BBBKR:DegreeSOn}, the result of Theorem \ref{thm: degformula} can be interpreted combinatorially in terms of non-intersecting lattice paths. We recall the Lindstr\"om--Gessel--Viennot Lemma (see e.g. \cite[Thm. 2.7.1]{Stanley:EnumerativeCombinatoricsVol1}):

\begin{lemma}[Lindstr\"om--Gessel--Viennot \cite{Lind:VectRepIndMatroids,GesselViennot}]\label{lemma: GV}
Let $A=\{a_1,\ldots,a_r\}$ and $B=\{b_1,\ldots,b_r\}$ be sets of points in $\mathbb{Z}^2$. Let $M_{i,j}$ denote the number of paths from $a_i$ to $b_j$ in the lattice $\mathbb{Z}^2$ using unit steps in only north and east directions. If the only way to connect all points in $A$ to all points in $B$ via non-intersecting paths is by connecting $a_i$ to $b_i$ then the number of ways to do this is given by $\det( [ M _{i,j}]_{i,j=1 \vvirg r})$.
\end{lemma}
\begin{example}\label{ex: St(4,6) nilps}
Consider the point configurations $A=\{(-3,0),(-2,0),(0,0)\}$, and  $B=\{(0,4),(0,2),(0,0)\}$. 
Then the matrix $M$ is given by 
\[
M={
\left[ \begin{array}{ccc}
\pbinom{7}{3} & \pbinom{5}{3} & \pbinom{3}{3} \\ \\
\pbinom{6}{2} & \pbinom{4}{2} & \pbinom{2}{2} \\ \\
\pbinom{4}{0} & \pbinom{2}{0} & \pbinom{0}{0}
\end{array} \right]} = 
\left[ 
\begin{array}{ccc}
35 & 10 & 1 \\ 
15& 6 & 1 \\ 
1 & 1 & 1
\end{array}\right].
\]
Its determinant is $44$. There is only one path from $A_3=(0,0)$ to $B_3=(0,0)$ and so a collection of non-intersecting lattice paths is uniquely determined by a pair of paths, one from from $A_1$ to $B_1$ and another from $A_2$ to $B_2$, not passing through $(0,0)$. 
\begin{figure}[!htpb]
\begin{center}
\includegraphics[scale=0.12]{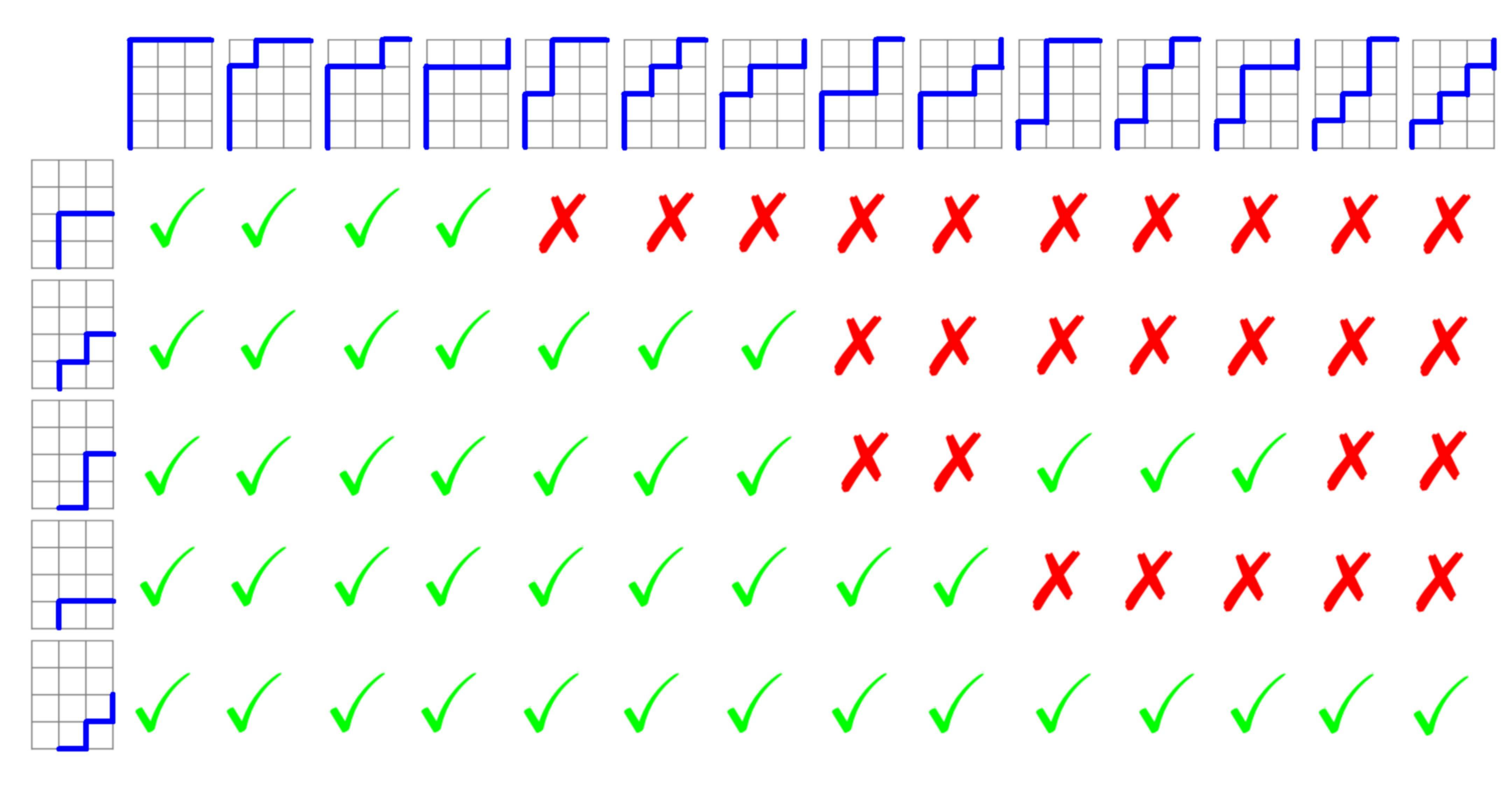}
\caption{All non-intersecting lattice paths (green) from  $A=\{(-3,0),(-2,0),(0,0)\}$ to $B=\{(0,4),(0,2),(0,0)\}$ }\label{fig: nilps}
\end{center}
\end{figure}

Figure \ref{fig: nilps} displays paths from $A_1$ to $B_1$ in the first row and paths from $A_2$ to $B_2$ in the first column. A green \checkmark indicates that the pair together with the stationary path at $(0,0)$ forms a collection of three non-intersecting lattice paths. Indeed, there are $44$ green {\checkmark}'s.
\end{example}

\begin{lemma}\label{lemma: GVmatrix is one from theorem}
Fix $k, n$ with $k+1 \leq n \leq 2k-1$. Let 
\begin{align*}
A &= \{ (- ( (\Omega_{k,n})_j + r - j),0) : j =1 \vvirg r\} \\ 
B &= \{(0,n-2j) : j =1 \vvirg r\}.
\end{align*}
The matrix in Theorem \ref{thm: degformula} is the matrix in the Lindstr\"om--Gessel--Viennot Lemma applied to $A$ and $B$. 
\end{lemma}
\begin{proof}
From the point $(-i,0)$ to $(0,j)$ there are $\binom{i+j}{i}$ paths. Notice that $n-2j=(\Omega_{n-1,n})_j+r-j$. These facts applied to $A$ and $B$ directly prove the result. 
\end{proof}

\begin{corollary}\label{corol: degreebylatticepaths}
For $k+1 \leq n \leq 2k-2$, let $L_{k,n}$ denote the number of non-intersecting lattice paths  
\[
\text{from } A=\{(-(\Omega_{k,n})_j-r+j,0)\}_{j=1}^r\text{ to }B=\{(0,n-2j)\}_{j=1}^r 
\]
 in $\mathbb{Z}^2$ consisting of unit steps in north and east directions. The degree of $\St{k}{n}$ is given by 
\[
\deg(\St{k}{n}) = 2^k \cdot L_{k,n}. 
\]
\end{corollary}
\begin{proof}
By Lemma \ref{lemma: GVmatrix is one from theorem}, the matrix in Theorem \ref{thm: degformula} is the matrix appearing in Lemma \ref{lemma: GV}. Apply Lemma \ref{lemma: GV} to the sets of $A$ and $B$ to conclude.
\end{proof}

\begin{example}[Degree of $\St{4}{6}$]\label{ex: degSt46 calculation}
Let $k=4$ and $n=6$.
 Example \ref{ex: St(4,6) nilps} calculated that $N(4,6)=44$. Applying Corollary \ref{corol: degreebylatticepaths} computes the degree of $\St{4}{6}$ to be $$\deg(\St{4}{6})=2^{4}\cdot 44=704.$$
\end{example}

\section{Conclusions}
The statements of Theorem \ref{thm: stiefel in bezout bound} ($n \geq 2k-1$) and Corollary \ref{corol: degreebylatticepaths} ($n \leq 2k-1$) combine to produce the proof of Theorem \ref{thm: main theorem}. We write it explicitly for completeness.

\begin{proof}[Proof of Theorem \ref{thm: main theorem}]
The first half of Theorem \ref{thm: main theorem} is given directly by Theorem \ref{thm: stiefel in bezout bound}. The second half is given by writing $\Omega_{k,n}$ in the point configuration in Corollary \ref{corol: degreebylatticepaths} according to its expression in \eqref{def: omega closed}.
\end{proof}

Theorem \ref{thm: main theorem} in the case $k=n-1$ gives the following corollary.

\begin{corollary}\label{cor: SO equals St for k n-1}
The degree of $\SO{n}$ is equal to the degree of $\St{n-1}{n}$. 
\end{corollary}

We provide a geometric proof of this fact as well.

\subsection{A geometric argument for the result of Corollary \ref{cor: SO equals St for k n-1}}
Consider the rational map
\[
\pi: \bbP (\Mat_{n \times n} \oplus \bbC) \dashto \bbP ( \Mat_{(n-1) \times n} \oplus \bbC) 
\]
sending an $n \times n$ matrix to the submatrix obtained by removing the first row. In other words, this is the projection with center $L = \{ (A,z) : z=0, A^{(i)} = 0 \text{ for } i > 1\}$, where $A^{(i)}$ denotes the $i$-th row of the $n \times n$ matrix $A$.

The restriction 
\[
 \phi: \bar{\SO{n}} \dashto \bbP ( \Mat_{(n-1) \times n} \oplus \bbC)
\]
surjects onto $\bar{\St{n-1}{n}}$. Since $\dim \SO{n} = \dim \St{n-1}{n}$, $\phi$ is generically finite.

We show that $\phi$ is regular. To see this, it suffices to show that $\bar{\SO{n}}$ does not intersect the center of the projection $L$. Suppose $(A,z) \in L\cap \bar{\SO{n}}$. In particular, $z = 0$ and $A$ is a matrix which is nonzero only in its first row and such that $AA^T = 0 \cdot \id_n = 0$. Notice that if $(A,z) \in \bar{\SO{n}}$, then $AA^T = A^TA$. This guarantees that if $A$ is supported on a single row and $AA^T = 0$, then $A = 0$ and we conclude that $\bar{\SO{n}} \cap L = \emptyset$. 

Moreover, $\phi$ is generically one-to-one. Indeed, let $B \in \St{k}{n}$ and consider $(B,1) \in \bar{\St{k}{n}}$, so that $BB^T = \id_{n-1}$. The rows of $B$ form a set of $n-1$ orthonormal vectors in $\bbC^{n}$; let $u$ be the unique vector in $\bbC^n$ that is orthogonal to the vectors of $B$, has norm equal to $1$ and forms a positively oriented basis together with the vectors of $B$. In particular, the matrix $A$ obtained by placing the vector $u$ above the matrix $B$ is an $n\times n$ orthogonal matrix with determinant $1$, and it is the unique preimage of $B$ via $\phi$. This shows $\deg \phi = 1$.

Applying iteratively \cite[Thm. 5.11(a)]{Mum:ComplProjVars}, we conclude
\[
 \deg \SO{n} = \deg \phi(\SO{n}) = \deg \St{n-1}{n}.
\]

\subsection{A final connection to the combinatorics of domino tilings}

The case $n=2k-1$ appearing as the overlap of Sections \ref{section: bezout range} and \ref{section: repdegree} produces the following simple combinatorial identity.
\begin{corollary}\label{corol: aztec}
\[
2^{\binom{r+1}{2}}=\det \left [ \binom{2i}{j} \right]_{i,j=1,\ldots,r} 
\]
\end{corollary}
\begin{proof}
When $n=2k-1$, the point configuration $A,B$ given by Lemma \ref{lemma: GVmatrix is one from theorem} has the property that the first $r-j+1$ steps beginning at $A_j$ must be vertical. Equivalently, the determinant of the path matrix associated to $A$ and $B$ is the same as the determinant of the path matrix associated to $\widetilde{A},B$ where $\widetilde{A} = \{(-(r-j+1),(r-j+1)\}_{j=1}^r$. The new path matrix is
\[
\mathcal P =  \left [ \binom{2i}{j} \right]_{i,j=1,\ldots,r}.
\]
We can express $\deg(\St{k}{2k-1})$ by Theorem \ref{thm: stiefel in bezout bound} as $2^{\binom{k+1}{2}}$ and by Theorem \ref{thm: degformula} as $2^k \det(\calP)$. We conclude
\[
 \det (\calP) = 2^{-k} \cdot 2^{\binom{k+1}{2}} = 2^{\binom{k}{2}} = 2^{\binom{r+1}{2}}.
\]

\end{proof}
We could only find the result of Corollary \ref{corol: aztec} in a comment in the sequence \texttt{A006125} in OEIS \cite{OEIS}. The \emph{Aztec diamond theorem} states that this power of two is the number of domino tilings of the Aztec diamond of order $n$. It was proved by Elkies, Kuperberg, Larsen, Propp in \cite{ElKuLaPr:AlternatingSignMatrices}. In \cite{EuFu}, Eu and Fu provide a proof of the Aztec diamond theorem using non-intersecting lattice paths, but they do not seem to use the path matrix in Corollary \ref{corol: aztec}. 

\section*{Acknowledgements} F.G. acknowledges financial support from the VILLUM FONDEN via the QMATH Centre of Excellence (Grant no. 10059). T.B. acknowledges financial support from the National Science Foundation (DMS-1501370). This collaboration began while the authors were visiting the Institute for Computational and Experimental Research in Mathematics in Providence, RI, during a semester long program on \emph{Nonlinear Algebra} in Fall 2018: we thank ICERM and the organizers for their support (NSF DMS-1439786) and for providing a wonderful research environment. We are grateful to Giorgio Ottaviani for his suggestions and comments on this project.

\bibliographystyle{amsalpha}
\bibliography{stiefel}

\end{document}